\documentclass[11pt, reqno]{amsart}
\usepackage{amsmath, bbm, amsthm, amssymb, amsfonts, enumerate, graphicx}
\usepackage[bookmarks=false]{hyperref}
\usepackage[usenames,dvipsnames]{color}
\usepackage{fullpage}
\def\Xint#1{\mathchoice
   {\XXint\displaystyle\textstyle{#1}}%
   {\XXint\textstyle\scriptstyle{#1}}%
   {\XXint\scriptstyle\scriptscriptstyle{#1}}%
   {\XXint\scriptscriptstyle\scriptscriptstyle{#1}}%
   \!\int}
\def\XXint#1#2#3{{\setbox0=\hbox{$#1{#2#3}{\int}$}
     \vcenter{\hbox{$#2#3$}}\kern-.5\wd0}}
\newcommand{\fint}{\Xint-}

\theoremstyle{definition}
\newtheorem{thm}{Theorem}

\newtheorem{lemma}[thm]{Lemma}
\newtheorem{cor}[thm]{Corollary}
\newtheorem{defi}[thm]{Definition}
\newtheorem{remark}[thm]{Remark}
\newtheorem{example}[thm]{Example}
\newtheorem{question}[thm]{Question}
\newtheorem{prop}[thm]{Proposition}
\newtheorem{claim}[thm]{Claim}
\numberwithin{thm}{section}

\newcommand{\id}{\operatorname{id}}
\newcommand{\E}{\mathbb E}
\newcommand{\R}{\mathbb R}

\newcommand{\Z}{\mathbb Z}
\newcommand{\N}{\mathbb N}
\newcommand{\mg}{\mathfrak{g}}
\newcommand{\Heis}{\text{Heis}}
\newcommand{\norm}[1]{\left\vert #1 \right \vert}	
\newcommand{\Norm}[1]{\left\Vert #1 \right \Vert}

\newcommand{\comment}[1]{}

\newcommand{\st}{\;:\;}

\newcommand{\V}{\mathcal V}
\newcommand{\HH}{G^{\perp}}
\newcommand{\cN}{\mathcal N}
\newcommand{\cT}{\mathcal{T}}

\newcommand{\cL}{\mathcal{L}}

\DeclareMathOperator{\NH}{NH}

\title{Separated Nets in Nilpotent Groups}
\author{Tullia Dymarz}
\address{Department of Mathematics \\ University of Wisconsin-Madison \\ 480 Lincoln Drive, Madison, WI 53706}
\email{dymarz@math.wisc.edu}
\author{Michael Kelly}
\address{Department of Mathematics \\ University of Michigan \\ 530 Church Street, Ann Arbor, MI 48109}
\email{michaesk@umich.edu}
\author{Sean Li}
\address{Department of Mathematics, The University of Chicago, Chicago, IL 60637}
\email{seanli@math.uchicago.edu}
\author{Anton Lukyanenko}
\address{Department of Mathematics \\ University of Michigan \\ 530 Church Street, Ann Arbor, MI 48109}
\email{Anton@Lukyanenko.net}
\thanks{T.D.\ was supported by NSF grant DMS-1207296.  S.L.\ was supported by NSF grant DMS-1600804. A.L.\ was supported by NSF RTG grant  DMS-1045119. M.K.\ was supported by NSF RTG grants DMS-1045119 and DMS-0943832.}

\begin{document}
\begin{abstract}In this paper we generalize several results on separated nets in Euclidean space to separated nets in connected simply connected nilpotent Lie groups. We show that every such group $G$ contains separated nets that are not biLipschitz equivalent. We define a class of separated nets in these groups arising from a generalization of the cut-and-project quasi-crystal construction and show that generically any such separated net is bounded displacement equivalent to a separated net of constant covolume.  In addition, we use a generalization of the Laczkovich criterion to provide `exotic' perturbations of such separated nets.
\end{abstract}
\maketitle 

\section{Introduction}
A subset $Y$ of a metric space $(X,d)$ is a {\it separated net} (or Delone set) if for some $0<c<C$, any two $c$-balls centered at distinct elements of $Y$ are disjoint (i.e. $Y$ is uniformly discrete), and the $C$-neighborhood of $Y$ is all of $X$ (i.e. $Y$ is coarsely dense). Separated nets are important in coarse geometry \cite{gromov1996geometric}, dynamical systems \cite{S1}, and mathematical physics \cite{BG}. In particular, in coarse geometry two spaces are considered equivalent (quasi-isometric) if they contain separated nets that are biLipschitz equivalent.

\begin{defi}[BL equivalence]A function $f : (X_1,d_{X_1}) \rightarrow ({X_2},d_{X_2})$ is $L$-biLipschitz if
\begin{align*}
  L^{-1} d_{X_1}(x,x') \leq d_{X_2}(f(x),f(x')) \leq L d_{X_1}(x,x') \qquad \forall x,x'  \in X_1.
\end{align*}
The metric spaces ${X_1}, {X_2}$ are BL equivalent if there exists a BL bijection between them.
\end{defi}

A fundamental question of Gromov \cite{gromov1996geometric} asks whether the choice of separated net matters: are any two separated nets $Y_1, Y_2$ in a metric space BL equivalent? 

A positive answer is known in some spaces: in $\R$, it is easy to find a BL map from any separated net to the integer lattice, and by Whyte's work in \cite{Whyte} any two separated nets in a \emph{non-amenable} metric space with \emph{bounded geometry}  (see \S \ref{sec:whyte} for definitions), are in fact BD equivalent:

\begin{defi}[BD equivalence]For two subsets $Y_1, Y_2 \subset (X,d)$ of a metric space, a function $f : Y_1\rightarrow Y_2$ is $L$-bounded-displacement for $C \geq 0$ if
\begin{align*}
   d(f(x),x) \leq C \qquad \forall x  \in Y_1.
\end{align*}
The subsets  ${Y_1}, {Y_2}$ are BD equivalent if there exists a BD bijection between them.
\end{defi}
For separated nets, BD equivalence implies BL equivalence but the converse is not necessarily true. Even in $\R^n$, is not hard to see that the separated nets $\Z^n$ and $(2\Z)^n$ are not BD equivalent. 
More surprisingly, for $n\geq 2$, not all separated nets in $\R^n$ are BL equivalent to $\Z^n$ \cite{BK,mcmullen1998lipschitz}. On the other hand, certain naturally-arising separated nets in $\R^n$, such as quasi-crystals arising from cut-and-project constructions in dynamics, generically do not exhibit this behavior \cite{BK2002, HKW}. 

In this paper, we generalize these results to separated nets in nilpotent Lie groups (assumed connected, simply connected, and equipped with a left-invariant Riemannian metric and Haar measure). Namely, for each such Lie group $G$ we provide a natural class of separated nets ($\Lambda$-nets) and show that a generic cut-and-project quasi-crystal in $G$ is BD equivalent to such a net. On the other hand, we generalize the results of Burago-Kleiner and McMullen to exhibit nets that are not BL equivalent to a $\Lambda$-net. 
Along the way we prove some auxiliary results that can be used to further our understanding of separated nets in $G$.


\subsection{Summary of Results}

\comment{
exp coords & def of l-nets
l-nets are nice, when BD to e.o. and lattices (use projection systems)
not BL to l-nets
Laczkovich-Whyte criterion
qcrystals are mostly l-nets
Laczkovich-Whyte criterion on dyadic tiles
exotic nets exist
}

Let $G$ be a nilpotent Lie group (always assumed to be connected and simply connected). In this setting, the exponential map from the Lie algebra $\mathfrak{g}$  to $G$ is a bijection, and so we can identify $G$ with $\R^n\simeq \mathfrak{g}$ with a group law of the form
$$(a_1,\ldots a_n)*(a_1', \ldots, a_n') = (a_1+a_1'+p_1, \ldots , a_n + a_n' + p_n)$$
where $p_i$ are polynomials in lower-indexed variables. These are known as \emph{exponential coordinates} on $G$. 
\begin{defi}
Let $\Lambda = (\lambda_1, \ldots, \lambda_n)$ have positive entries, and let $G$ be a nilpotent Lie group with exponential coordinates. By a \emph{$\Lambda$-net}, we mean the set $G(\Lambda)$ of points $(a_1, \ldots, a_n)\subset G$ such that $a_i \in \lambda_i \Z$ for each $i$. If each $\lambda_i=1$ we write $G(\Z)$. By \emph{$\Lambda$-box} we mean $[\lambda_1/2,\lambda_1/2)\times \ldots \times [-\lambda_n/2,\lambda_n/2)$. We denote the volume of the $\Lambda$-box by $\norm{\Lambda}$, and following the terminology for lattices, refer to it as the \emph{covolume} of $G(\Lambda)$ in $G$.
\end{defi}

In general $G(\Lambda)$ (and even $G(\Z)$) are not subgroups of $G$ (see \cite{Malcev} or \S \ref{sec:nilp} for more details) but they are always separated nets. 
The benefit of the $\Lambda$-net construction is that it gives us separated nets of all covolumes in any nilpotent Lie group.  In contrast, many nilpotent Lie groups do not have lattices and only some contain lattices of all covolumes \cite{Kuranishi}.  We call those that do contain lattices \emph{rational} nilpotent Lie groups. 

\begin{thm}
\label{thm:LambdaIsNice}
Every $\Lambda$-net is a  net and has an associated tiling by translates of the $\Lambda$-box. The nets $G(\Lambda_1)$ and $G(\Lambda_2)$ are BD equivalent if and only if $\norm{\Lambda_1}=\norm{\Lambda_2}$. Every lattice $\Gamma \subset G$ is BD to any $G(\Lambda)$ for which $\norm{\Lambda}$ is equal to the lattice covolume of $\Gamma$.
\end{thm}

This theorem follows from Lemma \ref{lemma:lambda-net} and an elementary argument (see \cite[Proposition 2.1]{HKW}).\\


\noindent{}{\bf BL-equivalence.} 
In \S \ref{sec:BL} we prove the following theorem:

\begin{thm}\label{thm:BLequiv}
Let $G\neq \R$ be a connected simply connected nilpotent Lie group and $Y_1\subset G$ a $\Lambda$-net. Then there exists a separated net $Y_2\subset G$ that is not BL equivalent to $Y_1$.
\end{thm}

The general outline of the proof of Theorem \ref{thm:BLequiv} follows \cite{BK} but there are some key difficulties that need to be addressed (see the introduction of \S \ref{sec:BL} for specifics). The main ingredient in the proof is producing a map that cannot be the Jacobian of a biLipschitz map from the asymptotic cone of $G$ to itself.

\begin{remark} Theorem \ref{thm:BLequiv} is also true for $Y_1$ an arbitrary separated net. On the other hand one could show that there are uncountably many biLipschitz equivalence classes of nets in $G$ (see \cite{magazinov} for the case of $G= \R^n$). Additionally, in this more general setting, $G$ may be taken to be any Lie group with polynomial growth (or even any locally compact groups of polynomial growth equipped with a left invariant path metric) since \cite{breuillard} shows that these groups have the same asymptotic cones as nilpotent Lie groups. Recent work \cite{DN16} proves results similar to Theorem \ref{thm:BLequiv} for separated nets in certain solvable Lie groups of exponential growth such as the three dimensional Lie group $SOL$. 
\end{remark}
\begin{question}\label{quest:main}\cite[Question 2]{BK2002} Are all lattices in a given nilpotent Lie group biLipschitz equivalent? Are all $\Lambda$-nets biLipschitz equivalent?
\end{question}
\noindent{}{\bf BD-equivalence.}
To analyze BD equivalence between arbitrary separated nets in a nilpotent Lie group $G$, in \S \ref{sec:BD} we generalize a theorem of Laczkovich by using a criterion of Whyte:
\begin{thm}[Laczkovich-Whyte Criterion]
Two nets $Y_1, Y_2 \subset G$ are BD iff there exists $C>0$ such that for any bounded measurable set $A\subset G$,
$$\norm{\#Y_1 \cap A - \#Y_2 \cap A} \leq C p(A)$$
where the coarse perimeter $p(A)$ is defined as the volume $\norm{N_1(A)}$ of the 1-neighborhood of the topological boundary of $A$. 
\end{thm}

 Next we focus on a class of naturally occurring separated nets, namely 
quasi-crystals arising from the cut-and-project construction (cf.\ \cite{hartnicketal}). In $\R^n$ these include well-studied nets such as lattices and the Penrose tiling, and are of substantial interest in solid state physics and crystallography \cite{BG, deBruijn,DO2,scoop}. We now extend the definition to the nilpotent setting:

\begin{defi}[Nilpotent cut-and-project quasi-crystals.]
Let $\phi: G \hookrightarrow G\times \R^m$ be a Lie group embedding giving an action of $G$ on  $G\times \R^m$ defined by $g\cdot (g', x)=\phi(g)*(g', x)$. Fix a parallelotope $S_0 \subset \R^m$ and set  $S=\id_G \times S_0$.  Then the projection of the integer points in $G\cdot S$ to $G$,
$$Y=Y(G, m, \phi, S_0):= \{ g\in G \st g\cdot S \cap G(\Z)\times \Z^m \neq \emptyset\}$$
is called a \emph{cut and project quasi-crystal}.
\end{defi}

\begin{remark}
If $G(\Z)$ is a lattice, then $\phi$ induces an action of $G$ on the nilmanifold $N=(G\times \R^m)/(G(\Z)\times \Z^m)$. One then has that the projection of $S$ to $N$ is a section for this action, and $Y$ is the set of return times to $S$. This interpretation is not available in the general case.
\end{remark}

The following result is a generalization of the main result of \cite{HKW} and says that most, but not all, quasi-crystals are BD equivalent to $\Lambda$-nets.
\begin{thm}
\label{thm:QC}
If $S_0$ is a box with sides parallel to the axes of $\R^m$ and $n\geq 2$, then for almost every embedding $\phi$, the set $Y(G, m, \phi, S_0)$ is a separated net BD to a $\Lambda$-net.
On the other hand, for almost every parallelotope $S_0$, every $m>0$, there is a residual set of embeddings $\phi$ such that  $Y(G,m,\phi, g_0, S_0)$ is a separated net \emph{not} BD to a $\Lambda$-net.
\end{thm}

\begin{remark}
We are not able to prove the corresponding result for BL equivalence since we are unable to prove a BL version of Lemma \ref{lemma:productnet} from \S \ref{sec:BD}. Such a result would immediately yield a generalization of the above theorem for the BL equivalence relation, as well as the BL equivalence of any two lattices in a given nilpotent Lie group thus answering Question \ref{quest:main}.
\end{remark}

The proof of Theorem \ref{thm:QC} relies on the corresponding result of Haynes-Kelly-Weiss in the Euclidean setting, which in turn appeals to a stronger form of the Laczkovich-Whyte criterion in $\R^n$. Defining \emph{dyadic tiles} (see \S \ref{sec:dyadic}) on a nilpotent Lie group and extending an efficient-counting result of Laczkovich (Theorem \ref{thm:efficientnilpotent}), we prove the corresponding criterion in $G$:

\begin{thm}
\label{thm:strongBD}
Let $Y_1, Y_2$ be nets in a rational nilpotent Lie group $G$ of topological dimension $n$.
Suppose that for every nilpotent dyadic tile $T$ of level $k$, there exists $\epsilon>0$ and $C>0$ such that
$$\norm{ \# T \cap Y_1 - \#T \cap Y_2 } \leq C 2^{k(n-1-\epsilon)}.$$
Then there exists a BD bijection $f: Y_1 \rightarrow Y_2$.
\end{thm}

\begin{remark}
While Theorem \ref{thm:strongBD} is stronger in the sense that it allows one to count integer points only on dyadic tiles rather than arbitrary bounded measurable subsets of $G$. On the other hand, it is only a sufficient condition, not a characterization of BD equivalence.
\end{remark}

\begin{remark} When $G$ is a rational \emph{Carnot} group (see \S \ref{sec:BL} for the definition) 
a corresponding theorem holds for \emph{Carnot dyadic tiles} (see Remark \ref{remark:CarnotDyadic}) with $(n-1)$ replaced by the homogeneous dimension of $G$ minus the length of its lower central series (i.e. step).
\end{remark}

We finish by considering the convenient fact that every $\Lambda$-net is a separated net both with respect to the left-invariant Riemannian metric $d_G$ on $G$ and with respect to the Euclidean metric $d_\E$ induced by exponential coordinates. 
\begin{defi}
\label{defi:exotic}
Let $d_1$ and $d_2$ be two metrics on a space $X$. We say that a set $Y\subset X$ is a \emph{($d_1$, $d_2$)-exotic net} if $Y$ is a separated net with respect to $d_1$ but is neither coarsely dense nor uniformly discrete with respect to $d_2$. When $X=\R^n$, $d_1=d_G$ and $d_2=d_\E$, we say that $Y$ is an \emph{exotic net}. 
\end{defi}

By reducing to the case of step-2 groups with one-dimensional center (i.e.\ Heisenberg groups, see Remark \ref{remark:Heisenberg}) and counting certain integer points in dyadic tiles, we prove:

\begin{thm}
\label{thm:exotic}
Let $G$ be a non-abelian nilpotent Lie group identified with $\R^n$ via exponential coordinates. Then every $\Lambda$-net $G(\Lambda)$ is BD equivalent to an exotic net.
\end{thm}

\begin{question}
Suppose $\R^n$ has two metrics $d_1$ and $d_2$ arising from two nilpotent Lie groups $G_1$ and $G_2$. Does there exist a ($d_1$, $d_2$)-exotic net in $\R^n$? Interesting cases include $G_1$ being isomorphic to $G_2$, or having the same asymptotic cone as $G_2$.  
\end{question}

\section{Nilpotent Lie groups}\label{sec:nilp}

We start by discussing the structure of nilpotent Lie groups, including some constructions that we believe to be new. See e.g.\ \cite{Malcev, GreenTao, breuillard} for more information.

\subsection{Exponential coordinates}
An $n$-dimensional Lie group $G$ with Lie algebra $\mathfrak g$ is called nilpotent of step $s$ if it has a finite lower central series
$$G=G_1 \supset G_2 \supset \ldots \supset G_s=1$$
where $G_i = [G, G_{i-1}]$ for $i>1$. From the corresponding nilpotency condition on $\mathfrak g$, it is easy to construct a \emph{Malcev basis} $\xi_1, \ldots, \xi_n$ for $G$, satisfying the condition that the linear span of  $\xi_i, \ldots, \xi_n$ is an ideal in $\mathfrak g$. The constants $s_{ijk}$ satisfying $[\xi_i, \xi_j] = \sum_k s_{ijk} \xi_k$ are known as the \emph{structural constants} of $\mathfrak g$.

As long as $G$ is connected and simply connected (which we always assume), the exponential map $\exp: \mathfrak g \rightarrow G$ is a diffeomorphism, and one can represent a point $g\in G$ in two types of \emph{exponential coordinates}:
\begin{align*}
g&= \exp(a_1 \xi_1 + \ldots + a_n \xi_n) = \exp(b_1 \xi_1)* \ldots *\exp(b_n \xi_n)
\end{align*}
The coordinates $(a_1, \ldots, a_n)$ of $g$ are \emph{exponential coordinates of the first kind}, while the coordinates $(b_1, \ldots, b_n)$ are \emph{exponential coordinates of the second kind}.

Using the Baker-Campbell-Hausdorff formula for $\exp$, one shows that the group law on $G$ in exponential coordinates is of the form
\begin{align*}
(a_1, \ldots, a_n) * (a_1', \ldots, a_n') = (a_1 + a_1'+p_1, \ldots, a_n+a_n'+p_n)\\
(b_1, \ldots, b_n) * (b_1', \ldots, b_n') = (b_1 + b_1'+q_1, \ldots, b_n+b_n'+q_n)
\end{align*}
where each $p_i$ is a polynomial in $(a_1, \ldots, a_{i-1}; a_1', \ldots, a_{i-1}')$ with real coefficients, and each $q_i$ is a polynomial in $(b_1, \ldots, b_{i-1}; b_1', \ldots, b_{i-1}')$ with real coefficients.

Recall that a \emph{lattice} in a Lie group is a discrete subgroup with finite covolume. Malcev famously established in \cite{Malcev} the connection between $G$ admitting a lattice and the coefficient types of the above polynomials.  
\begin{thm}[Malcev]
\label{thm:Malcev}
The following are equivalent for a nilpotent Lie group $G$:
\begin{enumerate}
\item $G$ admits a (uniform) lattice,
\item $\mathfrak g$ admits a Malcev basis with rational structural constants,
\item $\mathfrak g$ admits a Malcev basis such that the polynomials $p_i$ have rational coefficients,
\item $\mathfrak g$ admits a Malcev basis such that the polynomials $q_i$ have integer coefficients.
\end{enumerate}
\end{thm}

For a nilpotent Lie group satisfying the conditions of Theorem \ref{thm:Malcev} (called \emph{rational}), by \emph{integral exponential coordinates} we will mean exponential coordinates of the second kind such that the polynomials $q_i$ have integer coefficients.

\subsection{Projection systems}
\label{sec:projection}

Consider a nilpotent Lie group $G$ with (integral if appropriate) exponential coordinates $(x_1, \ldots, x_n)$ and group law
$$(x_1, \ldots, x_n) * (x_1', \ldots, x_n') = (x_1 + x_1'+p_1, \ldots, x_n+x_n'+p_n).$$

Define a group $G'$ as the space $\R^{n-1}$ with coordinates $(x_1, \ldots, x_{n-1})$ and group law
$$(x_1, \ldots, x_{n-1}) * (x_1', \ldots, x_{n'-1}) = (x_1 + x_1'+p_1, \ldots, x_{n-1}+x_{n'-1}+p_{n-1}).$$
  
Clearly, $G'$ is a nilpotent Lie group, the coordinates $(x_1, \ldots, x_{n-1})$ are (integral) exponential coordinates, and the projection $\pi: G \rightarrow G'$ given by $(x_1, \ldots, x_n) \mapsto (x_1, 
\ldots, x_{n-1})$ is a Lie group homomorphism, whose kernel $X_n$ lies in the center of $G$. 
Furthermore, the projection $\pi: G \rightarrow G'$ is a 1-Lipschitz map with respect to the induced Riemannian metric on $G'$. 

Repeating the projection yields a chain of nilpotent Lie groups of decreasing dimension and step, which we refer to as a \emph{projection system}. For any choice of $G$ of nilpotent step at least 2 we can ensure that the projection system eventually reaches a step-2 group, a step-2 group with one-dimensional center (if the exponential coordinates are ordered appropriately), and then $\R^m$ for some $m$ which need not be the dimension of the abelianization of $G$.

\subsection{$\Lambda$-nets}

A nilpotent Lie group need not have a lattice \cite{Malcev}, and may only have lattices of large co-volume \cite{Kuranishi}. We introduce \emph{$\Lambda$-nets} as a natural alternative.
\begin{defi}
Let $\Lambda = (\lambda_1, \ldots, \lambda_n)$ have positive entries, and let $G$ be a nilpotent Lie group with exponential coordinates (not necessarily integral). By a \emph{$\Lambda$-net}, we mean the \emph{set} $G(\Lambda)$ of points $(a_1, \ldots, a_n)\subset G$ such that $a_i \in \lambda_i \Z$ for each $i$. We write $\norm{\Lambda}:=\lambda_1 \cdot \ldots \cdot \lambda_n$.
\end{defi}

Generically, $G(\Lambda)$ is not a lattice. There are three important exceptions:
\begin{enumerate}
\item If $G$ is a abelian, then $G(\Lambda)$ always a lattice. 
\item If $G$ is non-abelian but rational and in integral exponential coordinates, and $\Lambda=(\lambda, \ldots, \lambda)$, then $G(\lambda \Z):=G(\Lambda)$ is a lattice if and only if $\lambda\in \N$. 
\item If $G$ is a rational Carnot group in integral exponential coordinates with dilation map $\delta_\lambda$ (see \S \ref{sec:BL} for definitions) then $G(\delta_\lambda (1, \ldots, 1))=\delta_\lambda G(\Z)$ is a lattice for any $\lambda>0$.
\end{enumerate}

\begin{lemma}
\label{lemma:lambda-net}
Every $\Lambda$-net is a separated net with an associated tiling by left-translates of a $\Lambda$-box $I_\Lambda=[\lambda_1/2,\lambda_1/2)\times \ldots \times [-\lambda_n/2,\lambda_n/2)$ of volume $\norm{\Lambda}$.
\begin{proof}
If $G$ is abelian, the lemma is obvious. 

If $G$ is non-abelian, we prove the lemma by induction on the dimension of $G$ using a projection system. Let $G'=G/X_n$ be the next group in the projection system, set $\Lambda'=(\lambda_1, \ldots, \lambda_{n-1})$ and assume by way of induction that the lemma is true for $G'(\Lambda')$.

Consider now the collection of left translates $G(\Lambda)*I_\Lambda$ of $I_\Lambda$ by elements of $G(\Lambda)$. Setting $X_n(\Lambda) = X_n\cap G(\Lambda)$, we have that the $X_n(\Lambda)$-translates of $I_\Lambda$ form a \emph{column} of boxes with disjoint interiors, and that this column projects to $I_{\Lambda'}$. More generally, the tiling $G(\Lambda)*I_\Lambda$ separates under the $X_n(\Lambda)$ action into a collection of columns, each of which projects to a tile in $G'$.  

Since the interiors of the projected tiles are disjoint in $G'$, the interiors of the columns are disjoint, and thus all the tiles $G(\Lambda)*I_\Lambda$ are disjoint. It is likewise clear that the tiles fill all of $G$, as desired.

Let $0<a < b < \infty$ be the minimal and maximal distance from $0$ to points of $\partial I_\Lambda$. We then have that $G(\Lambda)$ is an $(a, b)$-separated net.
\end{proof}
\end{lemma}

Theorem \ref{thm:LambdaIsNice} follows from Lemma \ref{lemma:lambda-net} by a standard argument (see \cite[Proposition 2.1]{HKW}).

\subsection{Coarse Perimeter}
\label{sec:coarseperimeter}

The \emph{coarse perimeter} $p(A)$ of a bounded measurable set $A\subset G$ can be defined in two ways, depending on a choice of $r>0$:
\begin{enumerate}
\item The volume $\norm{N_r(\partial A)}$ of the $r$-neighborhood of the topological boundary of $A$,
\item The number of points $\# Y$ in a maximal $r$-separated net $Y\subset \partial A$.
\end{enumerate}

We say that two notions $p_1$ and $p_2$ of perimeter are \emph{equivalent} if for some $C>0$ one has $p_1(A)\leq Cp_2(A)$ and $p_1(A)\leq Cp_2(A)$ for all bounded measurable sets $A$. It follows easily from the homogeneity of $G$ that that the two definitions of coarse perimeter above are equivalent and that their equivalence classes do not depend on the choice of parameter $r$ or maximal net $Y$. For example, fix a bounded measurable set $A\subset G$,
a maximal $1$-separated net $Y_1 \subset \partial A$ and a maximal $r$-separated net $Y_r \subset \partial A$ for $r>1$. Let $C$ be largest number of 1-separated points that fit inside a ball of radius $r$. Then $\#Y_1 \leq C \#Y_r$. The remaining inequalities follow analogously.

\begin{remark}
Note that coarse perimeter does \emph{not} agree with surface area for smoothly bounded sets $A$. In the next section we also define a notion, related to coarse perimeter, of a metric boundary of a discrete set.
\end{remark}

In rational nilpotent Lie groups, we will be interested in the perimeter of the unit cube $I_G=[-1/2, 1/2)^n$ and the perimeter of finite unions of its translates by the group $G(\Z)$. For such sets, we work with a combinatorial notion of perimeter that allows for inductive arguments on dimension (cf.\ Figure \ref{fig:dyadic}).

\begin{defi}
A \emph{face} of $I_G$ is a maximal non-trivial intersection $\partial I_G \cap g \partial I_G$ for some $0\neq g\in G(\Z)$. A choice of positive weights on the faces of $I_G$ induces a \emph{weighted perimeter measure} on the boundary of any union of tiles. If $A$ is a finite union of unit tiles, then we refer to the weighted perimeter measure as its \emph{combinatorial perimeter}.
\end{defi}

As in the Euclidean case, the boundary of the cube $I_G$ decomposes into \emph{horizontal} and \emph{vertical} faces:
\begin{lemma}
\label{lemma:faces}
The boundary of $I_G$ decomposes into finitely many faces of two types: horizontal (with constant $x_n$ coordinate) faces and vertical (non-constant $x_n$ coordinate). The projection of a horizontal face is all of $I_{G'}$ while the projection of a vertical face is a face of $I_{G'}$. Additionally the inverse image of such a face is a union of vertical faces.
\begin{proof}Finiteness follows from the compactness of $\partial I_G$ and the fact that $G(\Z)$ acts properly on $G$.

Every element $g\in G(\Z)$ has the form $g'* t$ for $g'\in G'$ and $t\in X_n$. Taking $g'=0$ gives the horizontal faces (top and bottom).  If a face is horizontal, then it is of the form $\partial I_G \cap g'* t * \partial I_G$ with $g'\neq 0$. The columns $X_n*I_G$ and $g'*X_n*I_G$ intersect along the boundary, giving a non-trivial intersection between $\partial I_{G'}$ and $g'*\partial I_{G'}$. Thus, each vertical face lives in the preimage of a face of $I_{G'}$.
\end{proof} 
\end{lemma}

\begin{lemma}
Fix a weighted perimeter measure on $I_G$ and let $A$ be a finite union of unit tiles. Then the coarse perimeter of $A$ is equivalent to the weighted perimeter measure of $\partial A$.
\begin{proof}
Since the choice of weighted perimeter measure does not affect its equivalence class, we may simply count faces. Let $F$ be the number of faces of $A$.

Fix $r>0$. Let $0<c_1<C_1$ be constants such that the $r$-neighborhood of any face of $I_G$ (viewed independently) has volume between $c_1$ and $C_1$. Let $C_2$ be the maximum number of faces intersecting any ball of radius $r$. Then the perimeter of $A$, computed as the volume of the $r$-neighborhood of $\partial A$, is bounded above  by $C_1 F$ and below by $(c_1/C_2) F$.
\end{proof}
\end{lemma}

\begin{defi}[Perimeter measure consistent with projections]
We will mostly work with combinatorial measures that are \emph{consistent with projections}. That is, for each projection in the projection system, we assume that each horizontal face has weight $1$, and that the weight of a face of $I_{G'}$ agrees with the total measure of its preimage in $I_G$.
\end{defi}

\begin{example}
Consider (see Figure \ref{fig:dyadic}) the projection $\pi: \Heis^1 \rightarrow \R^2$ of the three dimensional Heisenberg group to $\R^2$. In this case, $I_{\Heis^1}$ is the unit cube, projecting to the unit square $I_{\R^2}$ in $\R^2$. Each side of the square lifts to become three vertical faces of $I_{\Heis}$ (we should distribute a total weight 1 among these), and two more horizontal faces are added as usual (these each get weight 1). 
\end{example}

\subsection{Nilpotent dyadic tiles and efficient counting}\label{sec:dyadic}

Let $G$ be a rational nilpotent Lie group, viewed in integral exponential coordinates and $I_G$ the unit cube centered at the origin. By a \emph{tile} we will mean a translate of the unit cube $I_G$ by some element of $G(\Z)$. In addition, we define \emph{nilpotent dyadic tiles} on $G$ as follows:

\begin{defi}[Nilpotent dyadic groups and tiles]
Consider first the family of subgroups $\{G(2^i \Z) \st i\in \N\}$. For $i\geq 1$, $G(2^i \Z)$ has index $2^n$ in $G(2^{i-1}\Z)$. Set 
$$A_i := \{0, 2^{i-1}\}^n,$$ 
so that $G(2^i\Z)*A_i = G(2^{i-1}\Z)$ and $G(2^i\Z)*A_i*A_{i-1}*\ldots*A_1=G(\Z)$.

A \emph{discrete nilpotent dyadic tile of level $\ell$} is a subset $G(\Z)$ of the form
$$g*A_\ell*A_{\ell-1}*\ldots*A_1$$
for $g\in G(2^\ell \Z)$. 
A (continuous) nilpotent dyadic tile of level $\ell$ at $g\in G(2^\ell \Z)$ is the set
$$g*A_\ell*A_{\ell-1}*\ldots*A_1*I_G.$$
\end{defi}

\begin{figure}[h]
\centerline{\includegraphics[height=100px]{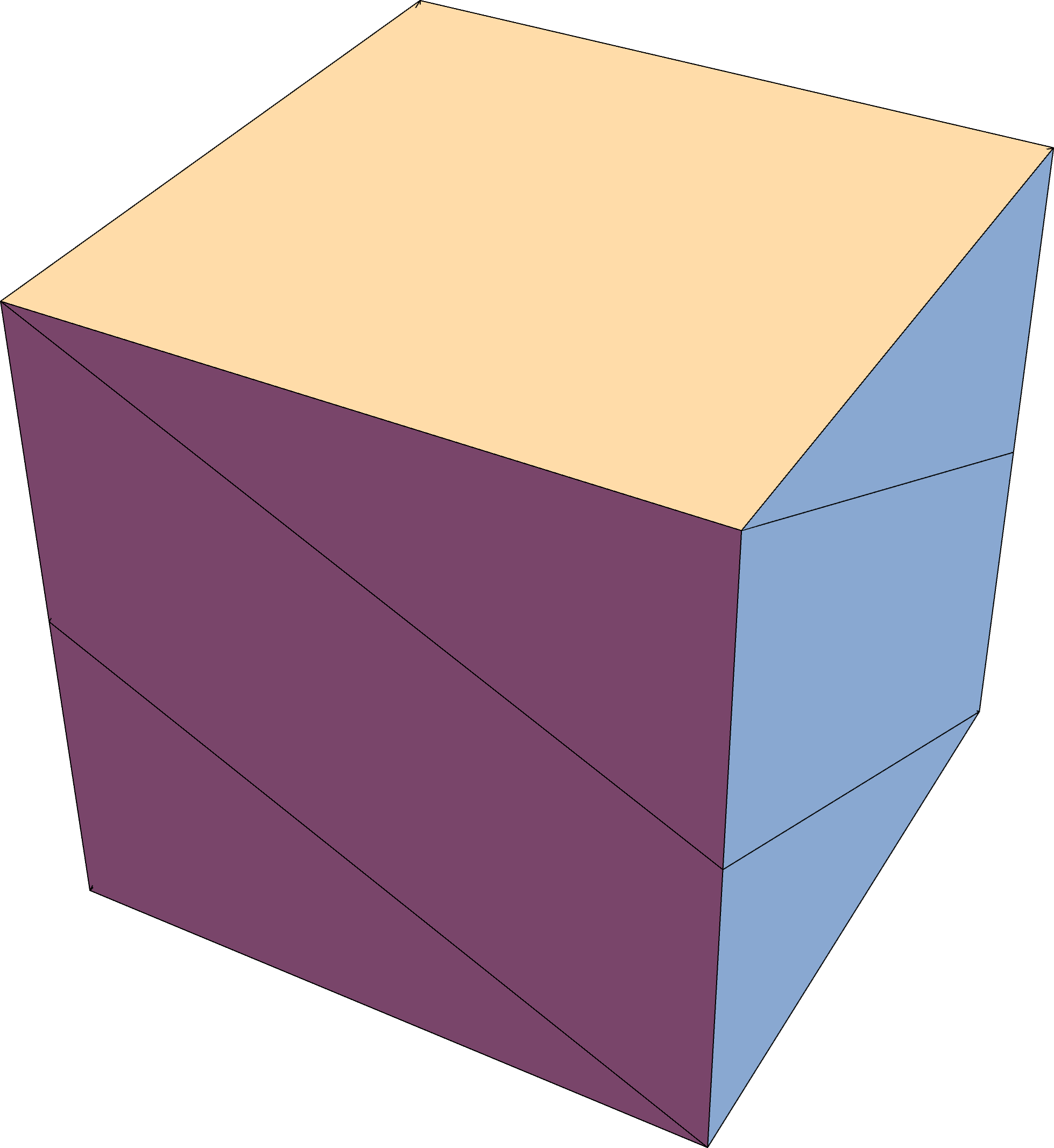} \includegraphics[height=100px]{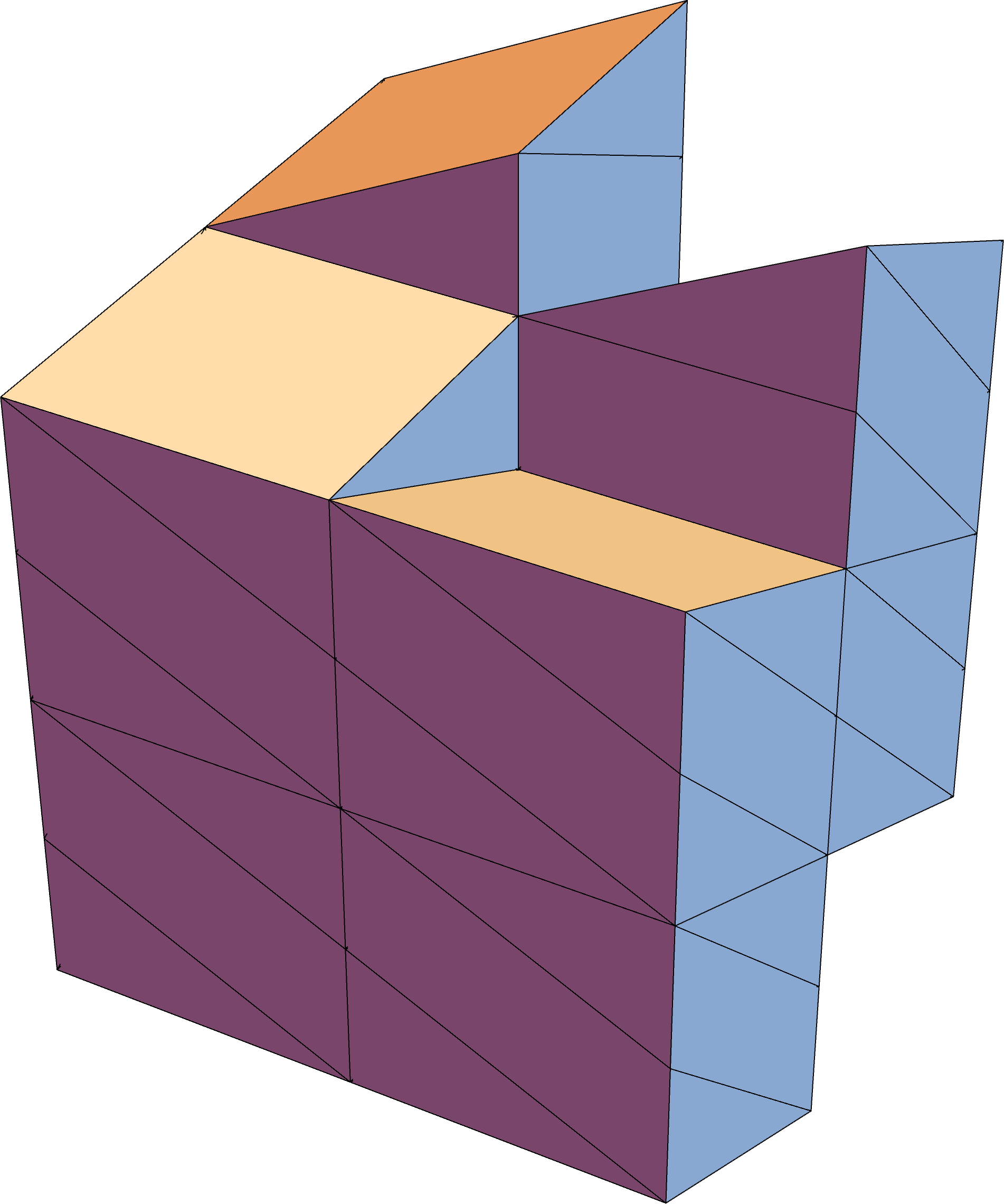} \includegraphics[height=100px]{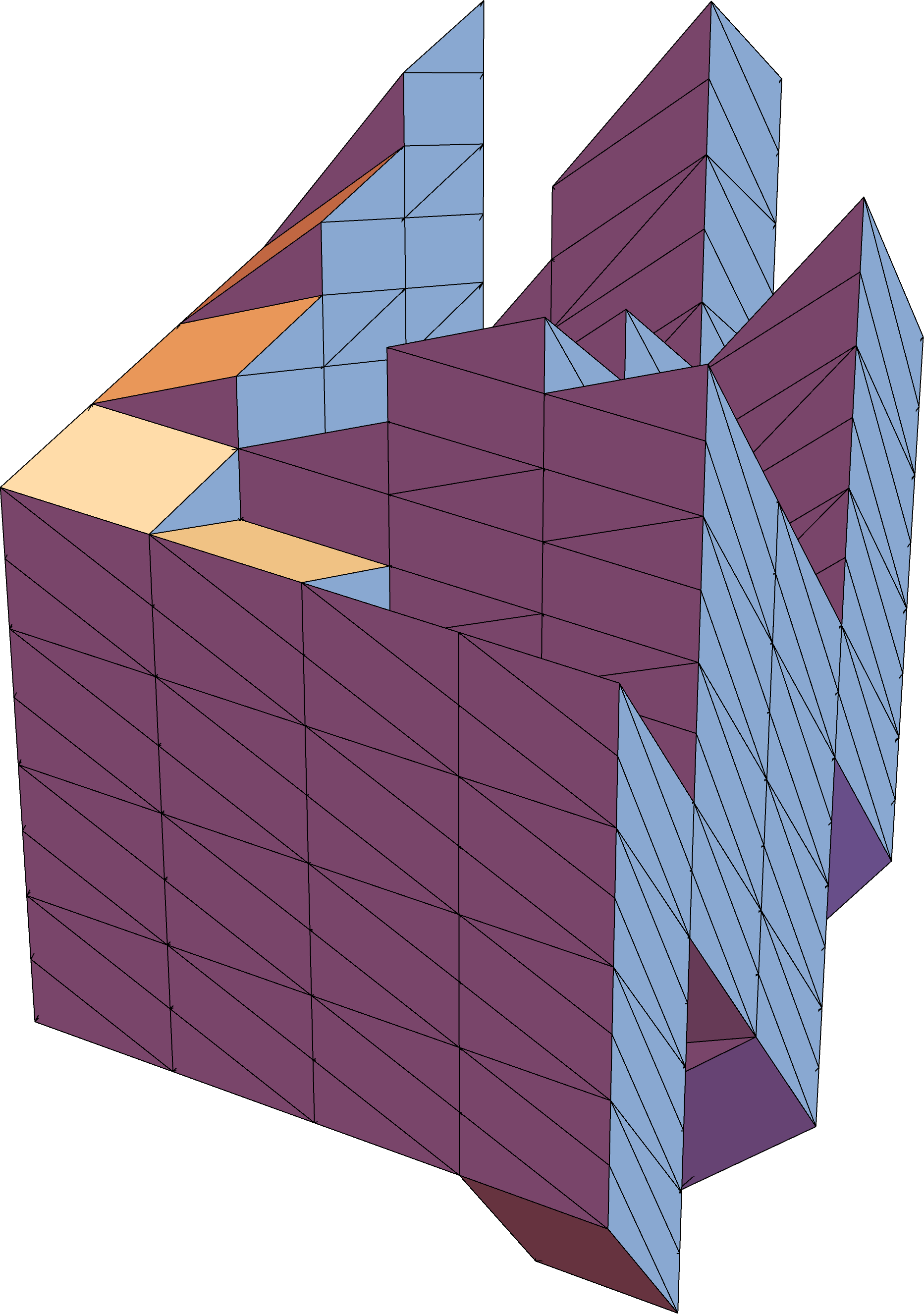}}
\caption{Nilpotent dyadic tiles of level 0, 1, 2 in the Heisenberg group (with polynomial $p_3=-0.5(xy'-yx')$ in the group law), rescaled using Euclidean dilation.}
\label{fig:dyadic}
\end{figure}

\begin{defi}
A finite union of tiles $A\subset G$ is \emph{describable} by a collection $\mathcal T$ of dyadic cubes if one can write $A$ using disjoint unions and proper differences involving elements of $\mathcal T$, with the restriction that each element of $\mathcal T$ should be used only once.
\end{defi}

It is clear that every finite union of tiles is describable using the unit tiles it contains, but a more efficient description is available if it has small perimeter. Consider, for example, the description ``the $1024\times1024$ square in $\R^2$ based at the origin with the top-right $1\times1$ square removed.''

Laczkovich provides in \cite{Laczkovich92UnifSpread} an algorithm for efficiently describing unions of tiles using dyadic tiles. While his proof is stated for $\R^n$, reading it with our definitions in mind provides:

\begin{thm}[Efficient Counting]
\label{thm:efficientnilpotent}
Let $G$ be a rational nilpotent Lie group of topological dimension $n$, viewed in integral exponential coordinates. Then there exists $C>0$ such that any finite union $A$ of tiles is describable by a collection of nilpotent dyadic tiles $\{T_i\}$ such that the number of tiles of each level satisfies
$$\#\{ T_i \st \ell(T_i)=k\} \leq C \frac{p(A)}{2^{k(n-1)}}.$$
\end{thm}

\begin{remark}\label{remark:CarnotDyadic}
If $G$ is a Carnot group (see \S \ref{sec:BL}), one can define a \emph{different} notion of dyadic tile based on the family of subgroups $\delta_{2^i}G(\Z)$. The $\delta$-rescaled dyadic tiles then limit to a tile defined by Strichartz in \cite{Strichartz}. For such dyadic tiles, Theorem \ref{thm:efficientnilpotent} is valid with $(n-1)$ replaced by $(Q-s)$ where $Q$ is the homogeneous dimension of $G$ and $s$ is its step. Surprisingly, the larger height of these Carnot dyadic tiles makes them unsuitable for our applications below.
\end{remark}

\section{Bounded Distance Perturbations}\label{sec:BD}

In  \cite{Laczkovich92UnifSpread} Laczkovich gives several criteria for when two separated nets in $\R^n$ are bounded distance from each other. The first criterion involves comparing the discrepancy between the separated nets on all bounded measurable sets relative to the size of their coarse perimeters. The second criterion involves the same estimate but now for sets that are disjoint unions of unit cubes. In this section we generalize 
both of Laczkovich's criteria where now unit cubes are replaced by the tiles
$G(\Z)*I_G$. 
 
\subsection{Whyte's Theorem}
\label{sec:whyte}
We start with Whyte's work \cite{Whyte} on bijective quasi-isometries between \emph{uniformly discrete bounded geometry spaces} (UDBG for short). Recall 
that a uniformly discrete space has bounded geometry if there is a function $\alpha: \R_+ \to \R_+$  such that the cardinality of all balls of radius $t$ is bounded by $\alpha(t)$. All separated nets in Lie groups are UDBG spaces but in this paper  we focus only on nilpotent Lie groups.  

Whyte's original work uses the language of uniformly finite homology but for our applications we can restate his theorem in the following way:

\begin{thm}[Whyte]\label{Whyte:thm} Suppose $Y_1, Y_2 \subset G$ are two separated nets in a nilpotent Lie group. Then $Y_1$ is bounded distance from $Y_2$ if and only if for all $r\gg 0$ there exists a $C\geq 0$ such that for all finite $ A \subset Y:=Y_1 \cup Y_2$ we have
\begin{equation}\label{whyte:eqn}  | \#(Y_1 \cap A) - \#(Y_2 \cap A)|  \leq  C \#(\partial_r A).\end{equation}
\end{thm}

Here $\partial_r A$ denotes the metric $r$-boundary of $A$ in $Y$. That is, $\partial_r A$ is one of the following sets (whose size is equivalent for sufficiently large $r$):
\begin{enumerate}
\item The outer boundary $N_r(A)\setminus A$,
\item The inner boundary $N_r(Y\setminus A)\cap A$,
\item The collar boundary $[N_r(Y)\setminus A] \cup [N_r(Y\setminus A) \cap A]$.
\end{enumerate}

\begin{remark}
Note that $Y$ is not necessarily uniformly discrete and hence not a UDBG space but we can make it one by a bounded perturbation of $Y_1$ that identifies a point of $Y_1$ with a point of $Y_2$ if they are less that $\epsilon$ apart where $\epsilon$ is chosen to be much less than the uniform discreteness constants of both $Y_i$.
 This perturbation can only change the cardinality of the boundary of any  set by a factor of $2$ and so does not affect the statement of the above theorem. 
\end{remark}

%

\subsection{Measurable Laczkovich-Whyte Criterion}

The goal of this subsection is to prove a version of Theorem \ref{Whyte:thm} with $A$ replaced by an arbitrary bounded measurable set and boundary replaced by coarse perimeter $p(\cdot)$ (see \S  \ref{sec:coarseperimeter} for the definition).

 \begin{thm}[Measurable Laczkovich-Whyte Criterion]\label{thm:LWMeas}
 Let $Y_1, Y_2$ be two separated nets in a nilpotent Lie group $G$. Then there exists a bijection $f:Y_1 \rightarrow Y_2$ satisfying the following:
 there exists $C_f\geq 0$ such that $d(y, f(y))<C_f$ for all $y\in Y_1$  if and only if there exists $C>0$ such that for all bounded measurable sets $A\subset G$,
\begin{equation}\label{discrep:eqn}\norm{\#(Y_1 \cap A) - \#(Y_2\cap A)} \leq  C p(A).\end{equation}
Furthermore, $C$ depends only on $G$ and the uniform discreteness constants of $Y_1$ and $Y_2$.
 \end{thm}

\begin{proof}
The proof follows from Theorem \ref{Whyte:thm} and the following two claims. 
\begin{claim}[Density is preserved]
\label{lemma:discrepancypreserved}
Let $Y_1, Y_2$ be two separated nets in a nilpotent Lie group $G$, and $f:Y_1 \rightarrow Y_2$ a bijection satisfying for all $y\in Y_1$ $d(y, f(y))<C_f$ for some $C_f\geq 0$. Then there exists $C>0$ such that for all measurable sets $A\subset G$,
$$\norm{\#(Y_1 \cap A) - \#(Y_2\cap A)} \leq  C p(A).$$
Furthermore, $C$ depends only on $G$ and the uniform discreteness constants of $Y_1$ and $Y_2$.
\begin{proof}
For each $i=0,1$, let $c_i$ be the uniform discreteness constant of $Y_i$. Let $C_i$ be the volume of a ball of radius $c_i$ in $G$, and $C=\operatorname{max}(C_1, C_2)$.

Since points in $A$ must stay near $A$ under $f$, we have
\begin{align*}
\#(Y_1 \cap A) &\leq \#(Y_2 \cap N_{C_f}(A))\  \leq \ \#(Y_2 \cap A) + \#(Y_2 \cap ( N_{C_f}(A) \setminus A))\\
&\leq \#(Y_2 \cap A) + \#(Y_2 \cap ( N_{C_f}(\partial A)))
\ \leq\  \#(Y_2 \cap A) + C \norm{N_{C_f}(\partial A)},
\end{align*}
which is sufficient by symmetry and equivalence of perimeters.
\end{proof}
\end{claim}

We next show that if Equation \ref{discrep:eqn} holds for all measurable sets in $G$ then Equation \ref{whyte:eqn} holds for all finite sets in $Y:=Y_1 \cup Y_2$. 

\begin{claim}[Measurable implies discrete]
Let $G$ be a nilpotent Lie group, and let $Y_1, Y_2 \subset G$ be separated nets.
Suppose that for all $r>0$  there exists $C_r>0$ such that for all bounded measurable sets $A\subset G$ one has
$$\norm{\#(Y_1 \cap A) - \#(Y_2\cap A)} \leq  C_r \norm{N_1(\partial A)}.$$
Then for all $r'>0$ large enough there exists $C'_{r'}>0$ such that for all finite sets $A\subset Y:=Y_1\cup Y_2$
 $$\norm{ \#(Y_1 \cap A) - \#(Y_2 \cap A)}  \leq  C'_{r'} \#(\partial_{r'} A).$$

\begin{proof}
  Let $R$ be the coarse density constant of $Y$ (i.e. the $R$ neighborhood of $Y$ is $G$). Suppose $F \subset Y$ is finite. 
 Then let $A= \bigcup_{y \in F} B_R(y)$ and note that $A \cap Y= F \cup \partial_R F.$
Also note that $$N_r(\partial A) \subset \bigcup_{y \in N_r(\partial A) \cap Y} B_q(y)$$ for $r,q$ large enough ($q>r> 2R$). 
Then
$$\norm{\#(A \cap Y_1) - \#(A \cap Y_2)} \leq C_r \norm{ N_r(\partial A)} \leq C_r|B_q(y)| \#( \partial_{r+R} F )$$ 
but 
 $$\norm{\#(F \cap Y_1) - \#(F \cap Y_2)} - \norm{\#(\partial_R F \cap Y_1) - \#(\partial_R F \cap Y_2)} \leq  \norm{\#(A \cap Y_1) - \#(A \cap Y_2)} $$
 so
  $$\norm{\#(F \cap Y_1) - \#(F \cap Y_2)} \leq \#(\partial_R F) + C_r|B_q(y)| \#( \partial_{r+R} F ) \leq K \#(\partial_{r+R} F)$$
 where $K= 1+ C_r|B_q|$.
\end{proof}
\end{claim}
This concludes the proof of Theorem \ref{thm:LWMeas}.
\end{proof}

\subsection{Laczkovich-Whyte Criterion with tiles}

Next we show that it suffices to check the Laczkovich-Whyte on tiles of bounded perimeter. Note that we continue to use the coarse perimeter $p(\cdot)$.

\begin{defi}[Bounded-geometry Tilings]
Let $G$ be a nilpotent Lie group and $\mathcal T$ a collection of pairwise disjoint measurable subsets (tiles) of $G$ such that $\sqcup \mathcal T = G$. We say that $\mathcal T$ is a \emph{bounded-geometry tiling of $G$} if the diameter of the tiles $T \in \mathcal T$ is uniformly bounded above and the volume of the tiles is uniformly bounded below.
\end{defi}

\begin{lemma}[Tile density implies measurable density]
\label{lemma:tilesareenough}\label{LaczTile:prop}
Let $\mathcal T$ be a bounded-geometry tiling of a nilpotent Lie group $G$. Let $Y_1, Y_2$ be separated nets and suppose that for any finite union of tiles $A'$, there exists $C_1>0$ such that
$$\norm{ \#(Y_1 \cap A') - \#(Y_2 \cap A')} \leq C_1 p(A').$$
Then, in fact for any bounded measurable set $A \subset X$ there exists $C_2>0$ such that
$$\norm{ \#(Y_1 \cap A) - \#(Y_2 \cap A)} \leq C_2 p(A).$$
\begin{proof} 
Let $A$ be a bounded measurable set in $G$, $A^+$ the finite union of tiles that intersect $A$ non-trivially, and $A^-$ the finite union of tiles that are contained in $A$. By containment, we have
\begin{align*}
 \#(Y_1 \cap A)  &\leq  \#(Y_1 \cap A^+)  \leq \#(Y_2 \cap A^+) + C_1p(A^+)\\
 &=\#(Y_2 \cap A) + \#(Y_2 \cap (A^+\setminus A))+ C_1p(A^+)\\
 &\leq\#(Y_2 \cap A) + \#(Y_2 \cap (A^+\setminus A^-))+ C_1p(A^+)
 \end{align*}
It thus suffices to bound both $ \#(Y_2 \cap (A^+\setminus A^-))$ and $p(A^+)$ by multiples of $p(A)$. 
  
Let $K$ be the upper bound on the diameter of tiles in $\mathcal T$. Then $\partial A^+$ is contained in the $K$-neighborhood of $\partial A$, so $p(A^+)$ is bounded by a multiple of $p(A)$. 
Likewise, $Y_2 \cap (A^+\setminus A^-)$  is contained in  $N_{K}(\partial A)$, with each tile  (and therefore each point) contributing a definite amount to the $\norm{N_K(\partial A)}$.  We thus have that $\#Y_2 \cap (A^+\setminus A^-)$ is bounded above by a multiple of $p(A)$, as desired. The opposite inequality follows by symmetry.
\end{proof}
\end{lemma}

\begin{defi}
A separated net $Y$ is \emph{uniformly spread with density $v$} (\emph{has density $v$})  if there exists $C>0$ such that for any bounded measurable set $A\subset G$, one has
$$\norm{\#(Y \cap A) - v^{-1} \norm{A} } < C p(A).$$
By Lemma \ref{lemma:tilesareenough} it suffices to check this condition on unions of tiles in a bounded-geometry tiling.
\end{defi}

\begin{defi}
A separated separated net $Y$ is \emph{of constant covolume $v$} if there exists a  bounded-geometry tiling $\mathcal T$ and a bijection $\phi: Y \rightarrow \mathcal T$ such that $y \in \phi(y)$ for each $y\in Y$, and furthermore each tile $T\in \mathcal T$ has volume $v$.
\end{defi}

\begin{lemma}[Density for constant covolume separated nets]
\label{lemma:constantCovolume}
Every separated net $Y$ of constant covolume $v$ has density $v$. In particular, every $\Lambda$-net has density $\norm{\Lambda}$.
\begin{proof}
Let $A^+$ and $A^-$ be as in Lemma \ref{lemma:tilesareenough}. Then
\begin{align*}
\#(Y \cap A) \leq \#(Y \cap A^+)  &=  v^{-1} \norm{A^+} \leq  v^{-1} \norm{A^-} +  v^{-1} \norm{A^+\setminus A^-}\\
& \leq  v^{-1} \norm{A} +   v^{-1} \norm{A^+\setminus A^-}
\end{align*}
As before, $  v^{-1} \norm{A^+\setminus A^-}$ is bounded by a multiple of $p(A)$, providing one direction of the inequality. The other direction follows analogously.
The claim about $\Lambda$-nets follows directly from Lemma \ref{lemma:lambda-net}.
\end{proof}
\end{lemma}

Combining the above results, we can now tell when a separated net is BD equivalent to a $\Lambda$-net.
\begin{thm}[Laczkovich-Whyte Criterion for $\Lambda$-net equivalence]
\label{thm:LWTiles}
Let $G$ be a nilpotent Lie group, and $Y$ a separated net. Then $Y_1$ is BD equivalent to a $\Lambda$-net $Y_2$ if and only if $Y$ is uniformly spread with density $\norm{\Lambda}$.
\begin{proof}
Lemma \ref{lemma:constantCovolume} gives that $Y_2$ has covolume $\norm{\Lambda}$, so the theorem follows form the Laczkovich-Whyte Criterion by the triangle inequality.
\end{proof}
\end{thm}

In certain cases, a nilpotent Lie group has lattices of every covolume, which allows us to make a stronger statement. In the case of rational Carnot groups (see \S \ref{sec:BL}), we obtain:
\begin{cor}
Suppose $G$ is a rational Carnot group. If a separated net $Y$ has density $v$, then it is BD to the lattice $\delta_{v^{1/Q}}G(\Z)$, where $Q$ is the homogeneous dimension of $G$ and $\delta$ the homogeneous dilation on $G$.
\end{cor}

We can now prove Theorem \ref{thm:strongBD}, which allows us to test for BD equivalence on only dyadic tiles. Note that unlike Theorem \ref{thm:LWTiles}, it is not a characterization of BD equivalence.

\begin{proof}[Proof of Theorem \ref{thm:strongBD}]
In order to apply Theorem \ref{thm:LWTiles}, let $A$ be an arbitrary union of unit tiles. Theorem \ref{thm:efficientnilpotent} describes $A$ using dyadic tiles $\mathcal D$ and the operations of disjoint union and set complement. The total discrepancy on $A$ is then bounded by the sum of the estimates on the tiles:
\begin{align*}
\norm{ \#{Y_1 \cap A} - \#{Y_2 \cap A}} &\leq \sum_{D\in \mathcal D} 2^{\ell(D) (n-1-\epsilon)} 
\leq \sum_{k=1}^\infty \#\{D\in \mathcal D \st \ell(D)=k\} \cdot 2^{k (n-1-\epsilon)} \\
&\leq C p(A) \sum_{k=1}^\infty 2^{-k(n-1)} 2^{k (n-1-\epsilon)} 
\leq C p(A) \sum_{k=1}^\infty 2^{-k(\epsilon)} \leq C p(A),
\end{align*}
where $C$ may change between occurrences but does not depend on $A$. Theorem \ref{thm:LWTiles} completes the proof.
\end{proof}

\subsection{Product nets}

A \emph{product net} is a separated net of the form $Y' \times \lambda \Z$, where $Y'\subset G'$ is a separated net and $\lambda>0$. Such separated nets arise naturally in our induction arguments and dynamical constructions below.

\begin{lemma}
\label{lemma:productnet}
Consider a nilpotent Lie group $G$ with a projection system (see \S \ref{sec:projection}). Let $\pi: G \rightarrow G'$ be given by $\pi(x_1, \ldots, x_n) = (x_1, \ldots, x_{n-1})$, let $\lambda>0$ and $Y'_1, Y'_2 \subset G'$ be two separated nets. Set $Y_1 = Y_1' \times \lambda \Z$ and $Y_2 = Y_2' \times \lambda \Z$. Then $Y_1, Y_2$ are separated nets, and there exists a BD bijection $f: Y_1 \rightarrow Y_2$ if and only if there exists a BD bijection $f': Y'_1 \rightarrow Y'_2$.
\begin{proof}

For each $y'\in Y'_1$, one has a column of points $(y', t)$ with $t\in \lambda \Z$, with columns related by left multiplication by an element of $G$. Each column is uniformly discrete since it admits a transitive $\lambda \Z$ action by isometries, and the columns are uniformly discrete and coarsely dense since the same holds for $Y'_1\subset G'$. The same argument works for $Y_2$.

Suppose first that there exists a bijection $f': Y'_1 \rightarrow Y'_2$ and that $\Norm{y'^{-1}*f'(y)}<C$ for some $C$. Set $f(y) = y*(y'^{-1}*f'(y), 0)$, which takes columns of integer points to columns of integer points and is therefore a bijection. We then have 
\begin{align*}
d(f(y), y) &= d(y*(y'^{-1}*f'(y), 0), y) = \Norm{(y'^{-1}*f'(y), 0)} < C,
\end{align*}
so that $f$ is a BD bijection. 

Conversely, suppose there exists a BD bijection $f: Y_1 \rightarrow Y_2$. In order to apply Theorem \ref{thm:LWMeas}, let $A' \subset G'$ be a bounded measurable set. Note first that the set $A_\infty:= A' \times \R$ has infinite perimeter $\norm{N_1(\partial (A'\times \R))}$, but that $\lambda \Z$ acts by isometries on $N_1(\partial (A'\times \R))$, and the quotient has some finite volume $v(A)$. For the finite subset $A_i := A' \times[-i\lambda, i\lambda]$, we then have the estimate $p(A_i) = 2i v(A) + C_i$ with some eventually-constant $C_i$ that can be interpreted as the perimeter of the top and bottom caps of $A_i$.

Now, using the fact that $Y_1$ and $Y_2$ are product nets and Theorem \ref{thm:LWMeas}, we obtain
\begin{align*}
\norm{ \# A' \cap Y'_1 - \# A' \cap Y'_2} &=
i^{-1} \norm{ \# A_i \cap Y_1 - \# A \cap Y_2}  \\
&\leq C   i^{-1} p(A_i) \leq C   i^{-1} (2i v(A) + C_i),
\end{align*}
where, $C>0$ may change between occurrences and does not depend on $i$ or $A$.

To complete the proof, it suffices to show that $v(A)$ is bounded by a multiple of $p(A)$. Because we are working with Lebesgue measure, we have that $v(A) = \lambda \norm{\pi N_1(\partial A'\times \R)}_{G'}$. It thus suffices to show $\pi(N_1(\partial A'\times \R))\subset N_r(\partial A')$ for some $r$ depending only on $G$, but this follows from the fact that $\pi$ is Lipschitz.
\end{proof}
\end{lemma}

\subsection{Application: exotic nets}

\label{sec:exotic}

We now prove Theorem \ref{thm:exotic}, which is a direct corollary of:
\begin{thm}
\label{thm:not-a-net}
Let $G$ be a non-abelian nilpotent Lie group identified with $\R^n$ via exponential coordinates. Then $G$ contains an exotic net that is BD equivalent to $G(\Z)$.
\end{thm}

We start by proving the theorem for step 2 groups with one dimensional center.

\begin{remark}
\label{remark:Heisenberg}
 In this case, the group law of $G$ has the form
$$\label{eq:step2coefficients}(x_1, \ldots, x_n)*(x_1', \ldots, x_n') = \left(x_1+x_1',\ldots, x_{n-1}+x_{n-1}', x_n+x_n'+ \sum_{1\leq i,j<n} a_{ij} x_i x_j'\right),$$
with coefficients $a_{ij}$ satisfying $a_{ij}=-a_{ji}$. In particular, the sum $\sum_{1\leq i,j<n} a_{ij} x_i x_j'$ is a non-degenerate skew-symmetric bilinear form, and there exists a linear change of the first $(n-1)$ coordinates so that the only non-zero coefficients are $a_{2i-1, 2i}=1$ and $a_{2i, 2i-1}=-1$ for $i=1, \ldots, (n-1)/2$. That is, $G$ is necessarily a Heisenberg group of topological dimension $n$. In particular, it follows from the above normalization that $G$ is rational, and from the existence of dilations 
$$\delta_r(x_1, \ldots, x_n) \mapsto (r x_1, \ldots, rx_{n-1}, r^2 x_n)$$
that $G$ has a lattice of every co-volume.
\end{remark}

Our construction hinges on the following technical lemma.

\begin{lemma}
\label{lemma:shear}
Let $G$ be a nilpotent Lie group of step 2 with one-dimensional center in rational exponential coordinates. Let $\theta\in [0,1]$ be irrational and fix a positive function $f$. There exists a function 
$x_\theta(\ell, i)$, increasing in each $\ell$ and $i$, such that a nilpotent dyadic tile \emph{$D=D(g,\ell)$ has at most $2^{\ell(d-2)}$ integer points} in the interval
$[-i, i]$ along the $X_n$-axis, provided that $g=(x_1, \ldots, x_n)$ satisfies
$x_1>x_\theta(\ell, i)$ and $\norm{(x_1, \ldots, x_{n-1})-(x_1, \theta x_1, 0, \ldots, 0)}<f(i)$ (note that the last coordinate is not compared).
\end{lemma}
\begin{proof}Permuting coordinates if necessary, we may assume that the coefficient $a_{12}$ in \eqref{eq:step2coefficients} is non-zero, and that all $a_{ij}$ are integral.

Consider first the effect of left multiplication by $g_0=(x, \theta x, 0, \ldots, 0)$. For an arbitrary $(x_1', \ldots, x_n')$ we then have
\begin{align*}
\label{eq:step2}
(x, &\theta x, \ldots, 0)*(x_1', \ldots, x_n') = \left(x+x_1', \theta x+x_2', x_3, \ldots, x_{n-1}', x_n'+\sum_{1\leq j<n} (a_{1j} x x_j'+ a_{2j}  \theta xx_j')\right).
\end{align*}
That is, the product is a composition of Euclidean translation by $(x ,\theta x, 0, \ldots, 0)$ and a shear in the plane $P$ spanned by the $x_n$ direction and the non-zero vector 
$$\left(a_{11}+\theta a_{21}, a_{12}+\theta a_{22},  \ldots, a_{1,{n-1}}+\theta a_{2,{n-1}}, 0\right),$$
with the extent of the shear proportional to $x$.

Now, consider what left multiplication by $g$ does to the $2^{n\ell}$ integer points of the dyadic tile $D_0=\bigsqcup A_\ell * \cdots *A_{1} * I_G$. The inner product 
\begin{equation}
\label{eq:slab}
(x_1, \ldots, x_{n-1}) \cdot \left(a_{11}+\theta a_{21}, a_{12}+\theta a_{22},  \ldots, a_{1,{n-1}}+\theta a_{2,{n-1}}\right)
\end{equation}
takes finitely many values on the integer points of $D_0$, partitioning the integer points into \emph{slabs} of points for which  \eqref{eq:slab} gives the same value. Because $\theta$ is irrational, the number of points in each slab is bounded above by $2^{\ell(n-2)}$ (note that projection onto the first two coordinates yields a $2^\ell\times 2^\ell$ square of integer points, each of which has preimage of size $2^{\ell(n-2)}$, and along $(\theta a_{21}, a_{12})$ can see only one of these points).

Under left multiplication by $g_0$, the slabs slide along each other, and if we have $x$ sufficiently large (larger than some $x_\theta(\ell, i)$), any interval of size $2i$ along the $X_n$ axis intersects at most one slab, giving us at most $2^{\ell(n-2)}$ integer points. 

Since the shear caused by left multiplication by $g=(x_1, \ldots, x_n)$ varies linearly with the first $n-1$ coordinates and is not influenced by the last coordinate, the perturbation from $g_0$ to $g$ does not substantially alter the height of the slabs, so that if $g$ is close to $g_0$ in the first $n-1$ coordinates, it still has at most $2^{\ell(n-2)}$ integer points in any interval of size $2i$ along the $X_n$ axis.

Lastly, a minimal choice of $x_\theta$ is naturally increasing in $i$ and we may furthermore choose it so that it is increasing in $\ell$.
\end{proof}

\begin{lemma}
\label{lemma:exoticNet}
Let $G$ be a nilpotent Lie group of step 2 with one-dimensional center in rational exponential coordinates. Then $G$ contains an exotic net that is BD to $G(\Z)$.
\begin{proof}
Since $G$ is rational and has step 2, fix $\theta \in [0,1]$ irrational and let $x_\theta(\ell, i)$ be the function provided by Lemma \ref{lemma:shear} with $f(i)=i+2^{i}$. Let $x_E(i)=x_\theta(i, i)+2^{i^2}$.

For each $i$, let $E_i$ be a Euclidean ball of radius $i$ centered at $(x_E(i), \theta x_E(i), 0, \ldots, 0)$. We take $E=\sqcup E_i$ and $Y = G(\Z) \setminus E$.

It is clear that $Y$ is not a separated net in the Euclidean metric. We now claim it is BD equivalent to $G(\Z)$ in the $d_G$ metric. To this end, let $D$ be a nilpotent dyadic tile of level $\ell$ (note that we are in rational coordinates so that we may use dyadic tiles), 
i.e.\ 
$$D = g* \bigsqcup A_\ell * \cdots *A_{1} * I_G$$
for some $g=(x_1, \ldots, x_n) \in G(2^\ell \Z)$. We set $x_D = x_1$.

By the Strong BD Criterion, Theorem \ref{thm:strongBD}, it suffices to bound the discrepancy 
$$\norm{ \# D \cap G(\Z)  - \# D \cap Y} = \#D\cap E\cap G(\Z)$$
by a multiple of $2^{\ell(n-2)}$. 

Suppose $E_i$ is a Euclidean ball that intersects $D$. We consider two cases: either $D$ is large compared to $E_i$, or it is small.

Suppose first that $\ell<i$. Projecting onto the first coordinate, $D$ becomes the interval $[x_D, x_D+2^\ell]$ and $E_i$ the interval $[x_E(i)-i, x_E(i)+i]$, so that $D$ cannot intersect any other $E_j$ because of the growth rate of $x_E(i)$. Thus it suffices to count the integer points of $D\cap E_i$, or just the integer points of $D$ in the interval $[-i, i]$ along the $X_n$ coordinate. Now, projecting onto each of the first $(n-1)$ coordinates, one sees that we must also have
$\norm{(x_1, \ldots, x_{n-1}) - (x_E(i), \theta x_E(i), 0, \ldots, 0)}< i+2^\ell = i+2^{i}$. Thus, by the choice of $x_\theta$ we have that $D$ contains at most $2^{\ell(n-2)}$ points in the interval $[-i, i]$ along $X_n$, as desired.

Suppose now $\ell>i$. Then, by the discussion of the previous case, any $E_j$ intersecting $D$ also satisfies $\ell > j$. Each $E_j$ contains roughly $j^n$ integer points, and we bound $\# D \cap E$ by
$$\sum_{j<\ell} \#E_j\cap D \cap G(\Z) \leq \sum_{j<\ell} \#E_j \cap G(\Z) \sum_{j<\ell} j^n \leq \ell^{n+1}\leq C 2^{\ell(n-2)},$$
for some $C>0$ as desired.

We thus have that $Y$ is a separated net BD to $G(\Z)$ and not coarsely dense with respect to $d_\E$. We now modify it so that it is also not uniformly discrete with respect to $d_\E$.

Permuting coordinates if necessary, we may assume that in the group law \eqref{eq:step2coefficients} we have $a_{12}\neq 0$. Note that $Y$ contains all but finitely many integer points along the $x_1$-axis, and set
\begin{align*}
Y'&=Y\cup \{(i, 0, 0, \ldots, 0)*(0, -.5/(a_{12}i), 0, \ldots, 0, -.5)\}_{2 \leq i\in \mathbb{N}}\\
 &= Y \cup \{(i, 1/i, 0, \ldots, 0)\}_{2 \leq i\in \mathbb{N}}
\end{align*}
The first line makes it clear that $Y'$ is $d_G$-discrete and therefore a $d_G$-net, while the second that it is not $d_\E$-discrete. Furthermore, $Y'$ is BD to
$$Y\cup \{(i, 0,\ldots,  0)*(0, 0,\ldots, -.5)\}_{2 \leq i\in \mathbb{N}},$$
which in turn is clearly BD to $Y$. Thus, $Y'$ is the desired exotic net.
\end{proof}
\end{lemma}

We now remove the restriction that $G$ should have \emph{rational} exponential coordinates. Note that when working with arbitrary exponential coordinates we are not able to use dyadic tiles and the strong Laczkovich-Whyte criterion.
\begin{lemma}
\label{lemma:step2arbitrary}
Let $G$ be a nilpotent Lie group of step 2 with one-dimensional center in exponential coordinates. Then $G$ contains an exotic net that is BD to $G(\Z)$.
\begin{proof}
By Remark \ref{remark:Heisenberg}, $G$ has a lattice $\Gamma$ of the same co-volume as $G(\Z)$, and $\Gamma$ is BD to $G(\Z)$ by Theorem \ref{thm:LambdaIsNice}. By Lemma \ref{lemma:exoticNet}, $G$ has an exotic net that is BD to $\Gamma$, as desired. 
\end{proof}
\end{lemma}

Theorem \ref{thm:not-a-net} now follows by using projection systems and product nets:
\begin{proof}[Proof of Theorem \ref{thm:not-a-net}]
If $G$ satisfies Lemma \ref{lemma:step2arbitrary}, then we are done. Otherwise, we may project away a one-dimensional subgroup of the center to reduce its dimension, and repeat the process to arrive at a rational nilpotent Lie group $G'$ of step 2  with one-dimensional center (this may require a mild reordering of the coordinates). Composing the projections, we get a map $\pi: G \rightarrow G'$, which forgets some coordinates of $G$. 

Let $Y'\subset G'$ be the exotic net provided by \ref{lemma:step2arbitrary}, and set  $Y= \pi^{-1}(Y')\cap G(\Z)$. It is clear that $Y$ is neither coarsely dense nor uniformly discrete in the Euclidean metric. On the other hand, applying Lemma \ref{lemma:productnet} inductively, we have that $Y$ is BD equivalent to $G(\Z)$ with respect to $d_G$, as desired.
\end{proof}

\subsection{Application: quasi-crystals}

\begin{proof}[Proof of Theorem \ref{thm:QC}]
We reduce the problem to studying quasi-crystals in Euclidean space.

Let $G'$ be the abelianization of $G$, with projection $\pi: G \rightarrow G'$. The embedding $\phi: G\hookrightarrow G\times \R^m$ then induces an embedding $\phi': G' \hookrightarrow G'\times \R^m$, and we can define $S' = \id_{G'} \times S_0 \subset G'\times \R^m$. Likewise, a choice of $g_0\in G\times \R^m$ gives a $g'_0 \in G'\times \R^m$ and an associated quasi-crystal $Y'\subset G'$.

Since $G'$ is isomorphic to $\R^d$ for some $d\geq 2$, Theorem 1.2 of \cite{HKW} then provides the desired statements for the set $Y'$. We now analyze the separated net $Y$. 

The embedding $\phi: G\rightarrow G\times \R^k$ induces, via projection onto the two factors, Lie group homomorphisms $\alpha: G\rightarrow G$ and $L: G \rightarrow \R^k$.  Because $\R^k$ is abelian, $L$ factors through $G'$, so that $L(g)$ is of the form $L'(g')$ where $g'=\pi(g)$.

Expanding out the definition of $Y$, we have
\begin{align*}Y&=\{g\in G \st \phi(g)*S \cap (G(\Z)\times \Z^m) \neq \emptyset \}\\
&=\{g\in G \st (\alpha(g), L(g))*(\id_G\times S_0) \cap (G(\Z)\times \Z^m)\neq \emptyset\}\\
&=\{g\in G \st \alpha(g) \in G(\Z) \text{ and } (L(g)+S_0) \cap \Z^m \neq \emptyset \}\\
&=\{g \in G \st \alpha(g) \in G(\Z) \text{ and } g'\in Y'\}.
\end{align*}
We now restrict to the generic case of $\alpha$ being an invertible map and inducing an invertible map $\alpha':G'\rightarrow G'$. We may then write
$$Y = \{ g \in G(\Z) \st g' \in \alpha'^{-1}(Y')\}.$$
Since $\alpha'$ is an invertible linear map, it is furthermore biLipschitz, and $\alpha'^{-1}(Y')$ is BD to a lattice if and only if $Y'$ is BD to a lattice.  Theorem 1.2 of \cite{HKW} then states whether or not $Y'$ is BD to a lattice, and the product net Lemma \ref{lemma:productnet} upgrades this to stating whether or not $Y$ is BD to a $\Lambda$-net.
\end{proof}

\begin{remark}
For nilpotent Lie groups admitting lattices of every covolume (e.g.\ rational Carnot groups, see \S \ref{sec:BL}), one may replace $\Lambda$-nets with lattices in Theorem \ref{thm:QC}.
\end{remark}

\begin{remark}
In contrast to the general theory of separated nets in nilpotent Lie groups (see \S \ref{sec:exotic}), we have that a quasi-crystal $Y\subset G$ is BD to a lattice in $G$ if and only if in exponential coordinates it is BD to a Euclidean lattice.
\end{remark}

\section{BiLipschitz Perturbations}\label{sec:BL}In this section, we prove Theorem \ref{thm:BLequiv}.  The proof will follow the strategy of \cite{BK}.  We first construct a measurable function $\zeta$ taking value in $\{1,1+c\}$ that cannot be the Jacobian of a bi-Lipschitz map.  We then construct an explicit net using this Jacobian and take an asymptotic limit of the supposed bi-Lipschitz homeomorphism.  The limiting map will be shown to have the Jacobian $\zeta$, contradicting the fact that $\zeta$ cannot be a Jacobian.  As asymptotic cones of simply connected nilpotent Lie groups are \emph{Carnot groups}, we will construct this non-Jacobian for Carnot groups.

\begin{defi}[Carnot group] A Carnot group is a simply connected nilpotent Lie group $G$
 with Lie algebra $\mg= \bigoplus_{j=1}^r \V_j$ where $[\V_i, \V_j] \subset V_{i+j}$.
 \end{defi}
 
Here $\V_1$ is called the \emph{horizontal layer}. The \emph{horizontal elements} of $G$ are those group elements of the form $e^{\lambda v}$ where $v \in \V_1$. A \emph{horizontal line} is given by $ t \mapsto ge^{tv}$ for all $t \in \R$ and for some fixed $g \in G$.  We say that this horizontal line (or any subsegment thereof) is going in the direction of $\pm v$.  We endow $\V$ with $|\cdot|$, the Euclidean norm. The group $G$ carries a natural subRiemannian metric called the Carnot-Carth\'eodory metric (see for example \cite{Li}) but we also work with other metrics on $G$. 

In this section, any $c_0 > 0$ will denote a constant---possibly changing instance to instance---that depends only on the group structure of $G$.  Any other constant $c_1,c_2,...$ or $C_0,C_1,...$ will denote unique fixed constants.

We use exponential coordinates of the first kind to canonically identify elements of $G$ with $\R^n \cong \mg$ via the exponential map $\exp$.  The \emph{standard dilation} on $G$ is given by 
\begin{align*}
  \delta_\lambda: G &\to G \\
  \exp(v_1+ v_2 +\cdots + v_r) &\mapsto \exp(\lambda v_1+ \lambda^2 v_2 +\cdots + \lambda^r v_r)
\end{align*}
when $v_i \in \V_i$.

We can define a projection onto the horizontal layer
\begin{align*}
  \pi: G &\to \V_1 \\
  \exp(v_1 + ... + v_r) &\mapsto v_1
\end{align*}
where $v_i \in \V_i$.  Assuming we have a metric on $G$, we can then measure how non-horizontal a group element $g$ is via the function $\NH(g)=d(\exp(\pi(g)), g)$.  Note that $\NH(a^{-1}b) = 0$ if and only if $a$ and $b$ lie on a horizontal line.  If $\NH(a^{-1}b) = 0$, we can define the horizontal interpolant
\begin{align*}
  \overline{ab} := \{ a \delta_t(a^{-1}b) : t \in [0,1] \}.
\end{align*}

We will use the following special semi-metric from \cite{Li}.
\begin{prop}[Proposition 7.2 of \cite{Li}] \label{p:convex-norm}
  There exists a continuous homogeneous norm $\|\cdot\| : G \to [0,\infty)$ and $p,C \geq 1$ so that if we define the semi-metric $d(g,h) = \|g^{-1}h\|$, then the following properties hold:
  \begin{enumerate}[(i)]
    \item for all $u,v,w \in G$ we have
    \begin{align}
      \frac{d(u,v)^p + d(v,w)^p}{2} \geq \left( \frac{d(u,w)}{2} \right)^p + C\NH(u^{-1}w)^p, \label{e:p-convex}
    \end{align}
    \item $\pi : (G,d) \to (\V_1,|\cdot|)$ is 1-Lipschitz.
    \item $\|g\| = |\pi(g)|$ if and only if $g = \exp(v)$ for some $v \in \V_1$,
  \end{enumerate}
\end{prop}

The semi-metric $d$, which we now fix, is only guaranteed to satisfy a quasi-triangle inequality.  That is, there exist $C_0 \geq 1$ so that
\begin{align}
  d(x,z) \leq C_0(d(x,y) + d(y,z)), \qquad \forall x,y,z \in G. \label{e:quasi-triangle}
\end{align}
However, we have the following estimate, which follows immediately from the continuity of the norm.
\begin{lemma} \label{l:quasi-triangle}
  There exists a continuous increasing function $\zeta : [0,1/2) \to \R$ with $\zeta(0) = 0$, so that if $g \in G$ and $d(g,h) < \epsilon d(0,g)$, then
  \begin{align*}
    d(0,h) \leq (1+\zeta(\epsilon))d(0,g).
  \end{align*}
\end{lemma}

Recall from \cite{Li}  Lemma 3.7 we have the following 
\begin{lemma} \label{l:drift}
There exists a constant $C_1 > 0$ that depends only on $G$ such that for $\eta \in (0,1)$ if $g,h \in G$ with $d(g,h) \leq \eta$ and $u \in \V_1$ with $|u| \leq 1$ then for $t \in [0,1]$
$$d (g e^{tu}, h e^{tu}) \leq C_1 \eta^{1/r}$$
where $r$ is the step of $G$. 
\end{lemma}

This holds for any (semi-)metric induced by a homogeneous norm.  The next theorem gives a quantitative bound for how non-horizontal a unit element of $G$ can be given how large its horizontal component is.

\begin{lemma} \label{l:NH-gtr}
  For all $\epsilon > 0$, there exists $\delta(\epsilon) > 0$ so that for all $g \in G$,
  \begin{align}
    \frac{\NH(g)}{\|g\|} < \delta &\Longrightarrow \frac{|\pi(g)|}{\|g\|} > 1-\epsilon, \label{e:NH-gtr} \\
    \frac{|\pi(g)|}{\|g\|} > 1-\delta &\Longrightarrow \frac{\NH(g)}{\|g\|} < \epsilon. \label{e:NH-less}
  \end{align}
\end{lemma}

\begin{proof}
  Let us first prove \eqref{e:NH-gtr}.  Suppose the statement is not true.  Then there exists $\epsilon > 0$ and $g_n \in G$ so that $\|g_n\| = 1$, $\NH(g_n) \to 0$, and $|\pi(g_n)| \leq 1 - \epsilon$.  As the unit sphere is compact, we may pass to a converging subsequence $g_n \to g$.  As $\NH$ and $\pi$ are continuous, we get that $\|g\| =1$, $\NH(g) = 0$, and $|\pi(g_n)| \leq 1-\epsilon$.  However, $\NH(g) = 0$ implies that $g$ is horizontal in which case we have by construction of $\|\cdot\|$ that $\|g\| = |\pi(g)|$, a contradiction.

  The proof of \eqref{e:NH-less} follows by a similar limiting argument.
\end{proof}

\begin{lemma} \label{l:uniform-equiv}
  The restriction of $\exp: (\V,|\cdot|) \to G$ on any compact set is a uniform homeomorphism.  The modulus of uniform continuity can depend on on the compact set.
\end{lemma}

\begin{proof}
  We have that $\exp$ is a homeomorphism.  Thus, restricting to compact sets, we get uniform continuity in both directions.
\end{proof}

As we have identified $G$ with $\R^n$ via exponential coordinates, we can talk about the Lebesgue measure $\cL^n$ on $G$.  It is not hard to see (by looking at the Jacobians of left translation and dilation) that $\cL^n$ is Haar and $\cL^n(\delta_\lambda (E))=\lambda^Q\cL^n(E)$ for $E \subset G$ measurable where
$Q= \sum_{i=1}^n i \cdot \dim(\V_i)$.
From now on, we use $|\cdot|$ to also denote the usual Lebesgue measure.  While, this does conflict with $|\cdot|$ as the Euclidean norm on $\V$, it should be clear from context which one we are using.  We will use $\cL^n$ if we want to be careful.

Given a vector $v \in \V_1$, we let $\pi_v : G \to \R$ denote the composition of $\pi$ with the projection onto the line spanned by $v$ and $\HH_v$ denote the exponential image of the subspace of $\mg$ orthogonal to $v$.  We will sometimes use an ordering on $\R v \subset \V$ as induced by its identification with $\R$ (choosing one of the two orderings arbitrarily).  Note that any horizontal line $g \exp(\R v)$ intersects $\HH_v$ in one unique point.  Thus, we can then define the following ``backwards projection'' function
\begin{align*}
  \nu_v: G &\to \HH_v \\
  g &\mapsto g \exp(\R v) \cap \HH_v.
\end{align*}

\begin{lemma}
  For any $v \in \V_1$ and $A \subseteq G$, we have
  \begin{align*}
    \cL^{n-1}(\nu_v(gA)) = \cL^{n-1}(\nu_v(A)), \qquad \forall g \in G.
  \end{align*}
  In particular, there exists a constant $c_1 > 0$ depending only on $v$, $G$, and its metric so that
  \begin{align*}
    \cL^{n-1}(\nu(B(g,r))) = c_1 r^{Q-1}, \qquad \forall g \in G, r > 0.
  \end{align*}
\end{lemma}

\begin{proof}
  We have
  \begin{align*}
    \cL^{n-1}(\nu_v(gA)) &= \cL^{n-1}(\{z \in \HH_v : z \cdot \exp(\R v) \cap gA \neq \emptyset\}) \\
    &= \cL^{n-1}(\{z \in \HH_v : g^{-1} z \cdot \exp(\R x) \cap A \neq \emptyset\}) \\
    &= \cL^{n-1}(g \{y \in \HH_v : y \cdot \exp(\R x) \cap A \neq \emptyset\} \exp(-\pi_v(g)v)) = (*).
  \end{align*}
  It follows easily from looking at the Jacobian of the Baker-Campbell-Hausdorff formula that both left and right multiplications $(\HH_v,\cL^{n-1}) \to (g\HH_v,\cL^{n-1})$ and $(\HH_v,\cL^{n-1}) \to (\HH_v g,\cL^{n-1})$ are measure preserving on the orthogonal subspaces.  Thus,
  \begin{align*}
    (*) = \cL^{n-1}(\{y \in \HH_v : y \cdot \exp(\R x) \cap A \neq \emptyset\}) = \cL^{n-1}(\nu_v(A)).
  \end{align*}
  We then have that $\cL^{n-1}(\nu_v(B(g,r))) = \cL^{n-1}(\nu_v(B(0,r)))$ and the second statement follows easily from a homogeneity argument with respect to dilation.
\end{proof}

\subsection{Constructing a non-Jacobian}
We now construct a measurable function taking value in $\{1,1+c\}$ that is not the Jacobian of any bi-Lipschitz function $G \to G$.  The proof will now resemble the proof in \cite{BK} where we construct a highly oscillatory function on a ``horizontal tube'' ($R_\lambda$ below), find line segments that must be quantitatively stretched, and nest the functions to compound the stretching.

There will be two twists in the Carnot setting.  The first is that the method of \cite{BK} only works if the endpoints of the tube are close to lying on a horizontal line segment.  This is because we compare how the function behaves on horizontal line segments lying in the horizontal tube $R_\lambda$ and we need the function to behave very horizontally to make the triangle inequality arguments work.  See case 1 below.  It turns out, however, that if the endpoints of the tube are very non-horizontal, then the needed quantitative stretching will follow easily from the properties of the metric in Proposition \ref{p:convex-norm}.  See case 2 below.

The second twist is that parallel horizontal line segments can drift apart in nonabelian Carnot groups (see Lemma \ref{l:drift}).  This will be particularly relevant when we have to compare volumes of neighboring regular balls.  To combat this, we will have to take much thinner slices of the tube (compared to the radius) and compare balls in neighboring slices.  As these balls are much smaller than the radius, we have to look at a much denser set of horizontal lines going through the tube.  By using the fact that the behavior of Lipschitz functions propagate locally (by uniformly continuity) we can actually look at {\it every} horizontal line going through $R_\lambda$.  This allows us to replace the pigeonhole principle used in \cite{BK} with an averaging method.

We now begin the construction.  We will assume $G \neq \R$.  Without loss of generality, we equip the $G$ in the target with the semi-metric from Proposition \ref{p:convex-norm} (remember only the quasi-triangle inequality holds).  This is possible as the statement we are trying to prove is invariant under bi-Lipschitz remetrization, and the metric of Proposition \ref{p:convex-norm} is bi-Lipschitz to the subRiemannian metric.  For the domain $G$, we equip that with any left-invariant homogeneous true metric (for example, the subRiemannian metric).

Fix an orthonormal basis of $\mg$ and pick $x\in \V_1$ in this basis. We call $e^{tx}$ the $x$-axis 
and write $\HH = \HH_x$ and $\nu = \nu_x$.
Note that $\pi_x(\HH)=0$.
We also set $d_{\pi_x}(g,h)=|\pi_x(h^{-1}g)|$.

For any $\lambda > 0$, we define $\bar{B}_\lambda := \HH \cap B(0,\lambda)$.  We can also define the tube
$$R_\lambda= \{ be^{tx} \mid b \in \bar{B}_\lambda,  t \in [0,1] \}$$
and break $R_\lambda$ into $N$ equal `slices'
$$R^i_\lambda = \{ be^{tx} \mid b \in \bar{B}_\lambda,  t \in [(i-1)/N,i/N) \}$$
for $i=1,...,N$.  We can add the face $\bar{B}_\lambda e^x$ to $R^N_\lambda$.  A simple homogeneity argument gives that $\cL^{n-1}(\bar{B}_\lambda)=c_0\lambda^{Q-1}$.  By Fubini, we also have $|R_\lambda| = c_0 \lambda^{Q-1}$.

By construction, the backwards projection of $R_\lambda$ is $\bar{B}_\lambda$.  We first claim that the backwards projection of a sufficiently small neighborhood of $R_{\lambda/4C_0}$ lies in $\bar{B}_{\lambda/2}$.

\begin{lemma} \label{l:claim-0}
  There exists $K_0 > 0$ so that if $N > K_0\lambda^{-r}$, then
  \begin{align}
    \nu(B(R_{\lambda/4C_0},1/100N)) \subseteq \bar{B}_{\lambda/2}. \label{e:rhoy}
  \end{align}
\end{lemma}

\begin{proof}
  Let $z \in R_{\lambda/4C_0}$, $y \in B(z,1/100N)$, and $\eta \in [-1/100N,1/100N]$ be so that if $z' = z e^{\eta x}$, then $y^{-1} z' \in G^\perp$.  Then by the quasi-triangle inequality, we have that
  \begin{align*}
    d(z',y) \leq C_0(d(z',z) + d(z,y)) \leq c_0 N^{-1}.
  \end{align*}
  Multiplying $z'$ and $y$ by some $e^{-\zeta x}$ so that $\zeta \in (0,2)$ and $z'e^{-\zeta x},ye^{-\zeta x} \in G^\perp$, we get from Lemma \ref{l:drift} and the fact that $N > K_0 \lambda^{-r}$ that
  \begin{multline*}
    d(ye^{-\zeta x},0) \leq C_0(d(ye^{-\zeta x},z'e^{-\zeta x}) + d(z'e^{-\zeta,x},0)) \leq C_0(C_1(c_0 N^{-1})^{1/r} + \lambda/4C_0) \\
    \leq C_0(c_0 K_0^{-1/r} \lambda + \lambda/4C_0).
  \end{multline*}
  Thus, by taking $K_0$ sufficiently large, we get \eqref{e:rhoy}.
\end{proof}


\begin{defi}
  We say that a pair of points $g_1, g_2 \in G$ is $A$-stretched by a map $f:G \to G$ if 
  $$d(f(g_1), f(g_2)) \geq A d(g_1,g_2).$$
\end{defi}

For any $\lambda, N > 0$, we can define an alternating function
\begin{align*}
  \rho_{\lambda,N} : R_\lambda &\to \{1,1+c\} \\
  z &\mapsto \begin{cases}
    1, & z \in R^i_\lambda, ~i \text{ is odd}, \\
    1 + c, & \text{otherwise}.
  \end{cases}
\end{align*}

We now show that a bi-Lipschitz function whose Jacobian oscillates no less than $\rho_{\lambda,N}$ on $R_\lambda$ must quantitatively stretch some horizontal pair of points more than how the endpoints of the tube are stretched.  See Lemma 3.2 of \cite{BK} for the analogous statement.

\begin{lemma} \label{l:stretch-incr-1}
  Let $L \geq 1$ and $c > 0$.  There exist $k >0, \lambda_0 > 0$, $\mu > 0$, and $K > K_0$ such that if $\lambda \in (0,\lambda_0)$, $N > K\lambda^{-r}$ and $f : R_\lambda \to G$ is a $L$-bi-Lipschitz map whose Jacobian $\rho$ satisfies
  \begin{align*}
    |\rho_{\lambda,N}^{-1}(1) \backslash \rho^{-1}([0,1])| &< \mu N^{-Q}, \\
    |\rho_{\lambda,N}^{-1}(1+c) \backslash \rho^{-1}([c,\infty))| &< \mu N^{-Q},
  \end{align*}
  then the following property holds: if $0, e^x\in G$ are $A$-stretched by $f$, then there exists some $\eta > 1/2$ and $g,ge^{\eta x/N}$ both in $R_{\lambda/2}$ that are $(1+k)A$-stretched.
\end{lemma}

\begin{proof}
Let $\rho$ be the Jacobian of the bi-Lipschitz map $f : R_\lambda \to G$.  Suppose no such pairs of points in $R_{\lambda/2}$ are $(1+k)A$-stretched for some $k > 0$ to be determined.  Without loss of generality we can assume that $f(0)=0$.  Let $y = \pi(f(e^x))$.

We say a point $g\in R_\lambda$ is $(\ell,A)$-regular if
$$\pi_y(f(ge^{\frac{1}{N}x})) - \pi_y(f(g)) >(1- \ell)\frac{A}{N}.$$
We say that a ball is \emph{irregular} if some point inside that ball is not regular.

Let $R'$ consists of the points in $R_{\lambda/4C_0}$ that lie near the center of each slice $R_\lambda^i$: 
$$R' = \bigcup_{i=0}^{N-2} \left\{ g \in R_{\lambda/4C_0}^i : |\pi_x(g)| \in \left(\frac{4i+1}{4N},\frac{4i+3}{4N} \right) \right\}.$$
Let $\cN$ be a $1/10N$-net of $R'$.

The proof will now follow from a series of claims inspired by Burago-Kleiner. 

\begin{claim} \label{c:claim-1}
  There exists $\tau_1 = \tau_1(\ell,L)$, $k_1=k_1(\ell,L)$, and $\lambda_1=\lambda_1(\ell,L)$ such that if $\lambda < \lambda_1$, $k<k_1$, $N > K_0 \lambda^{-r}$, and
  \begin{align*}
    \frac{\NH(f(e^x))}{\|f(e^x)\|} < \tau,
  \end{align*}
  then there exists a $1/100N$-ball centered at a point of $\cN$ that is regular.
\end{claim}

\begin{proof}
Assume all $1/100N$-balls around points of $\cN$ are irregular, so that there is some point (choosing one arbitrarily for each point of $\cN$) inside the ball that is not $(\ell, A)$-regular.  For $s = (\ell/100C_1L^2)^r$, consider an $s/N$-ball $B^i$ around each one of these chosen irregular points and call this a \emph{bad} ball.  Note from the triangle inequality that bad balls are all pairwise disjoint.

Let $I$ denote the index set of all bad balls $B^i = B(z,s/N)$.  As each ball of $I$ is uniquely associated to a point of a $1/10N$-net of $R'$, which has volume $|R'| \geq c_0 \lambda^{Q-1}$, we get that $\# I = \# \cN \geq c_0N^Q \lambda^{Q-1}$.

Define a function $\phi: \bar{B}_{\lambda/2} \to \R$
$$\phi(g)= \sum_{i \in I} \chi_{\nu(B^i)}(g)$$
By Lemma \ref{l:claim-0}, we have that each $\chi_{\nu(B^i)}$ is indeed supported in $\bar{B}_{\lambda/2}$.  Then since $\#I \geq c_0 N^Q \lambda^{Q-1}$, $\cL^{n-1}(\nu(B^i)) = c_1 (s/N)^{Q-1}$ and $\cL^{n-1}(\bar{B}_{\lambda/2}) \leq c_0 \lambda^{Q-1}$ we must have for some $g_0 \in \bar{B}_{\lambda/2}$ that 

$$\phi(g_0) \geq  \fint_{\bar{B}_\lambda} \phi(g)\ dg \geq c_0 s^{Q-1} N.$$ 

Let $\gamma = g_0 \exp([0,1]x)$ be the horizontal line segment.

We see that there must be at least $c_0 s^{Q-1}N$ cells $Q_i = R_\lambda^i \cap R'$ so that some point of $Q_i \cap \gamma$ is contained in a bad ball.  In particular, we may suppose there are more than $c_2 N$ of these cells $Q_i$ with (say) $i$ even where $c_2$ depends on $\ell$ and $L$.

Next we split segment $\gamma$ into subsegments given by the points
$$g_0,g_0e^{t_0x},g_0e^{t_1x},...,g_0e^{t_{N-1}x},g_0e^x$$
where
\begin{enumerate}[(i)]
  \item $t_i \in \left( \frac{4i+1}{4N}, \frac{4i+3}{4N} \right)$, \label{e:center}
  \item if possible choose $t_{2i}$ so that $g_0 e^{t_{2i}}$ is in a bad ball, in which case choose $t_{2i+1} = t_{2i}+ \frac{1}{N}$. \label{e:even}
\end{enumerate}
By the discussion above, \eqref{e:even} happens at least $c_2 N$ times.

Suppose $g_0e^{t_{2i}x}$ is in a bad ball.  Then it is within distance $s/N$ of an irregular point. Let $g_1=g_0e^{t_{2i}x}$ and $g_1'$ be the irregular point at distance $s/N$ away and let $g_2,g_2'$ be their right translates by $e^{\frac{1}{N}x}$.  Lemma \ref{l:drift} implies that since $d(\delta_N(g_1), \delta_N(g_1'))\leq s$ then $$d(\delta_N(g_1)e^x, \delta_N(g_1')e^x)\leq s^{1/r}$$ which implies that $$d(g_1e^{\frac{1}{N} x}, g_1'e^{\frac{1}{N}x}) =d(g_2, g_2')\leq C_1 \frac{s^{1/r}}{N}.$$  Remembering $s = (\ell /100C_1L^2)^r$, we get
\begin{multline}
d_{\pi_y}(f(g_1),f(g_2)) \leq d_{\pi_y}(f(g_1),f(g_1')) + d_{\pi_y}(f(g_1'), f(g_2')) + d_{\pi_y}(f(g_2), f(g_2')) \\
\leq \frac{Ls}{N} + (1- \ell) \frac{A}{N} + \frac{C_1Ls^{1/r}}{N} \leq \left(1 - \frac{\ell}{2} \right) \frac{A}{N}. \label{e:l/2-compress}
\end{multline}
In the last inequality, we used the fact that $L^{-1} \leq A$.

By \eqref{e:center}, one has that $|t_{i+1}-t_i| > \frac{1}{2N}$ for all $i$ and so we have by initial assumption of the lemma that
\begin{align}
  d(f(g_0e^{t_ix}),f(g_0e^{t_{i+1}x})) \leq (1+k)A (t_{i+1} - t_i), \qquad \forall i. \label{e:no-kstretch}
\end{align}
Finally, we have
\begin{align}
  \max\{d(f(g_0),f(g_0e^{t_1x})),d(f(g_0e^{t_Nx}),f(g_0e^x))\} \leq \frac{L}{N} \leq L \lambda, \label{e:endedges}
\end{align}
where we used the fact that $N > K\lambda^{-r} > \lambda^{-1}$ (for sufficiently small $\lambda$).

We now use another series of segments defined by the points
$$0, g_0, g_0e^{t_1x}, \ldots, g_0 e^{t_{N_1}x},g_0e^{x}, e^{x}$$
to estimate the distance between the projections of $f(0)=0$ and $f(e^x)$. 
Again by Lemma \ref{l:drift} we have that since $d(0,g_0)\leq \lambda $ then  $d(e^x,g_0e^x) \leq C_1 \lambda^{1/r}$, and so
\begin{align}
  \max\{d(f(0),f(g_0)),d(f(e^x),f(g_0e^x))\} \leq C_1 L \lambda^{1/r}. \label{e:endpoints}
\end{align}

Remembering that there are at least $c_2 N$ even indices $i$ so that $g_0e^{t_ix}$ is in a bad ball, we get the following upper bound:
\begin{align*}
  |\pi_y(f(e^x))| \overset{\eqref{e:l/2-compress} \wedge \eqref{e:no-kstretch} \wedge \eqref{e:endedges} \wedge \eqref{e:endpoints}}{\leq} c_2 N\left( 1- \frac{\ell}{2} \right) \frac{A}{N} + \left(1- c_2 \right)(1+k) A + 4C_1 L \lambda^{1/r}.
\end{align*}
Note that we do not suffer a quasi-triangle inequality loss as $(g,h) \mapsto |\pi_y(g) - \pi_y(h)|$ satisfies the triangle inequality.

By taking $k$ and $\lambda$ sufficiently small depending on $\ell$ and $L$ (remember $c_2$ depends also on $\ell,L$), we get
\begin{align*}
  |\pi_y(f(e^x))| \leq \left(1 - \frac{\ell}{4}\right) A.
\end{align*}

On the other hand, by choosing $\tau$ in the hypothesis small enough depending on $\ell$, we can get from \eqref{e:NH-gtr} that
\begin{align*}
  |\pi_y(f(e^x))| > \left(1- \frac{\ell}{4} \right) \|f(e^x)\| \geq \left(1- \frac{\ell}{4} \right) A,
\end{align*}
a contradiction.
\end{proof}

Let $W= \exp(\pi(f(e^x))/N)$.
For any point $g$ let $W_g:= f(g)^{-1}f(ge^{\frac{1}{N}x})$. 

\begin{claim} \label{c:claim-2}
Given any $m>0$ there is an $\ell_2=\ell_2(m)>0$ such that if $\ell \leq \ell_2$ and $k < \ell$ then for every $(\ell, A)$-regular point $g \in R_{\lambda/2}$ so that $ge^{x/N} \in R_\lambda$ we have $d(W, W_g) < \frac{m}{N}$.
\end{claim}

\begin{proof}
  Remember $y = \pi(f(e^x))$.  We have that
  \begin{align*}
    |\pi(W_g)| &\geq |\pi_y(W_g)| > \frac{(1-\ell)A}{N}, \\
    \|W_g\| &\leq \frac{(1+k)A}{N} \leq \frac{(1+\ell)A}{N}.
  \end{align*}
  The second inequality comes from the assumption that $g,ge^{x/N} \in R_{\lambda/2}$ are not $(1+k)A$-stretched.  Thus, by \eqref{e:NH-less}, if we take $\ell$ small enough, we get
  \begin{align*}
    NH(W_g) < \frac{m}{2C_0L(1+\ell)} \|W_g\| \leq \frac{m}{2C_0N}.
  \end{align*}
  Note that $W = \exp(\pi(W))$.  As $N \pi(W)$ and $N \pi(W_g)$ both lie in a ball of radius $L$, we have by Lemma \ref{l:uniform-equiv} that
  \begin{align*}
    Nd(W,\exp(\pi(W_g))) = d(\exp(N \pi(W)),\exp(N\pi(W_g))) < \frac{m}{2C_0}
  \end{align*}
  if
  \begin{align}
    |\pi(W) - \pi(W_g)| < \frac{m'}{N} \label{e:small-dev}
  \end{align}
  for some sufficiently small $m'$ depending only on $m$.  As $W_g$ is $\ell$-regular, we get that the projection of $\pi(W_g)$ onto $\pi(W) = y/N$ is at least $(1-\ell)\frac{A}{N}$.  On the other hand,
  \begin{align*}
    |\pi(W_g)| \leq \|W_g\| \leq \frac{(1+k)A}{N}.
  \end{align*}
  Thus, we get from Claim 2 of \cite{BK} that \eqref{e:small-dev} is possible if we take $\ell$ sufficiently small depending on $m'$.

  We then have that by quasi-triangle inequality that
  \begin{align*}
    d(W,W_g) \leq C_0(d(W,\exp(\pi(W_g))) + NH(W_g)) < C_0\left(\frac{m}{2C_0N} + \frac{m}{2C_0N} \right) = \frac{m}{N}
  \end{align*}
\end{proof}

\begin{claim} \label{c:claim-3}
  There exist $m_3 = m_3(c,L) >0$ such that if $m<m_3$ and $S$ is a $1/100N$-ball so that $d(W,W_g)< \frac{m}{N}$ for every point $g \in S$, then for $S' := S \cdot e^{x/N}$ we have
  $$\left||f(S)| - |f(S')| \right| <\frac{c}{2} |S|.$$ 
\end{claim}

\begin{proof}
  Let $S = B(z,1/100N)$, $Q = f(S)$, $R = f(S')$, and $\tilde{R} = Q \cdot W$.  By translation invariance of the measure, we have that $|\tilde{R}| = |Q|$.

  Let $g \in \partial S$.  Then by hypothesis, we have that $d(f(g) \cdot W, f(ge^{x/N})) < m/N$ and $g \cdot e^{x/N} \in \partial S'$.  We thus have a bound on the Hausdorff distance
  \begin{align*}
    d_{\text{Haus}}(\partial \tilde{R}, \partial R) \leq m/N.
  \end{align*}
  Thus, it suffices to calculate the volume $|B(\partial R,m/N)|$.

  Consider a $m/LN$-net $\mathcal{M}$ of $\mathcal{A} = (B(z,1/100N) \backslash B(z,1/100N - m/LN)) e^{x/N}$.  This is a thickened interior edge of $S'$.  As $|B(x,r)| = c_0r^Q$, we have that $|\mathcal{A}| \leq c_0 \frac{m}{L} N^{-Q}$, and so $\#\mathcal{M} \leq c_0 (m/L)^{1-Q}$.  As $f$ is $L$-bi-Lipschitz,
  $$f(\partial S') \subseteq B(f(\mathcal{M}),m/N).$$
  By the quasi-triangle inequality, we get
  $$B(\partial R,m/N) = B(f(\partial S'),m/N) \subseteq B(f(\mathcal{M}),2C_0m/N),$$
  and so
  \begin{align*}
    |B(\partial R,m/N)| \leq \# \mathcal{M} \cdot c_0 \left( \frac{2C_0m}{N} \right)^Q \leq c_0 \frac{mL^{Q-1}}{N^Q}.
  \end{align*}
  Thus, by taking $m$ small enough depending on $c$ and $L$ we get that
  \begin{align*}
    \left| |f(S')| - |f(S)| \right| = \left| |R| - |\tilde{R}| \right| < \frac{c}{2} |S|.
  \end{align*}
  Here, we used the fact that $|S|$ is comparable to $N^{-Q}$.
\end{proof}

Given $c,L > 0$, we use Claim \ref{c:claim-3} to get $m_3 > 0$.  We then use Claim \ref{c:claim-2} to get $\ell_2$ from $m_3$ and finally Claim \ref{c:claim-1} to get $\tau$, $\lambda_1$, and $k_1$ from $\ell_2$.  Now choose any $\lambda < \lambda_1$ and $N > K\lambda^{-r}$ to construct $\rho_{\lambda,N}$.  We may choose $\tau$ sufficiently small (depending only on $p$) so that
\begin{align}
  (1 + 2^pC\tau^p)^{1/p} \geq 1 + 2^{p-1}Cp^{-1} \tau^p, \label{e:taylor}
\end{align}
where $C$ is the constant from Proposition \ref{p:convex-norm}.  We now fix $k = \min(2^{p-1}Cp^{-1} \tau^p, k_1)$.

{\bf Case 1:} Suppose $\NH(f(e^x))/\|f(e^x)\| < \tau$.

  By Claim \ref{c:claim-1}, there is a $1/100N$-ball $B$ around a point of $\cN$ that is regular.  By definition of $R'$, there is some $i \in \{0,...,N-1\}$ so that $B \subset R_\lambda^i$, which means $B e^{x/N} \subset R_\lambda^{i+1}$.  Thus, $\rho_{\lambda,L}$ take different valued constants on $B$ and $Be^{x/N}$ (we may suppose without loss of generality that $\rho_{\lambda,N} = 1$ on $B$).  Then as $\rho^{-1}([0,1])$ differs from $\rho_{\lambda,N}^{-1}(1)$ on a set of measure no more than $\mu N^{-Q}$, we get that
  \begin{align*}
    |f(B)| &\leq |B| - \mu N^{-Q} + L^Q \mu N^{-Q}.
  \end{align*}
  Similarly, we have that
  \begin{align*}
    |f(Be^{x/N})| &\geq (1+c)(|B| - \mu N^{-Q}).
  \end{align*}
  As $|B|$ is comparable to $N^{-Q}$, we get from choosing $\mu$ to be sufficiently small that
  \begin{align*}
    |f(Be^{x/N})| - |f(B)| \geq c |B| - (c+L^Q)\mu N^{-Q} \geq \frac{c}{2} |B|.
  \end{align*}
  On the other hand, by Lemma \ref{l:claim-0}, $B$ (and so also $Be^{x/N}$) is contained in $R_{\lambda/2}$.  Then a combination of Claim \ref{c:claim-2} and Claim \ref{c:claim-3} gives us that
  \begin{align*}
    \left||f(B)| - |f(Be^{x/N})| \right| < \frac{c}{2} |B|,
  \end{align*}
  a contradiction.

{\bf Case 2:} Suppose $\NH(f(e^x))/\|(f(e^x)\| \geq \tau$.

  Consider the points $f(0),f(e^{x/2}),f(e^x) \in G$.  By the hypothesis of the current case, we have that
  \begin{align*}
    \frac{d(f(0),f(e^{x/2}))^p + d(f(e^{x/2}),f(e^x))^p}{2} - \left( \frac{d(f(0),f(e^x))}{2} \right)^p \overset{\eqref{e:p-convex}}{\geq} C \tau^p d(f(0),f(e^x))^p.
  \end{align*}
  Multiplying by $2^p$ on both sides, we get
  \begin{align*}
    \frac{(2d(f(0),f(e^{x/2})))^p + (2d(f(e^{x/2}),f(e^x)))^p}{2} &\geq d(f(0),f(e^x))^p + 2^p C \tau^p d(f(0),f(e^x))^p \\
    &\geq (1+2^pC\tau^p) A^p.
  \end{align*}
  In the second inequality, we used that $d(f(0),f(e^x)) \geq A d(0,e^x) = A$.  In particular, one of the two terms in the numerator is larger than the right hand side.  Without loss of generality, suppose it is the first term.  Then we get
  \begin{multline*}
    d(f(0),f(e^{x/2})) \geq (1 + 2^pC\tau^p)^{1/p}A d(0,e^{x/2}) \overset{\eqref{e:taylor}}{\geq} (1 + 2^{p-1}Cp^{-1} \tau^p) A d(0,e^{x/2}) \\
    \geq (1+k) A d(0,e^{x/2}).
  \end{multline*}
  Here, we used the fact that $d(0,e^{x/2}) = 1/2$.  This contradicts the initial assumption of the lemma.
\end{proof}

The following lemma is obvious.

\begin{lemma} \label{l:dense-segments}
  For any $\lambda > 0$, $\eta > 0$, and $\delta > 0$, there exists a finite collection of disjoint horizontal line segments $\overline{l_jr_j} \subset R_\lambda$ in the direction of $x$ so that for any $g,ge^{sx} \subset R_\lambda$ where $s > \eta$, there exists a horizontal segment $\overline{l_jr_j}$ so that $d(l_j,g) < \delta$ and $d(r_j,ge^{\eta x}) < \delta$.
\end{lemma}

\begin{lemma} \label{l:one-step}
  There exist $k' > 0$ such that, for any horizontal line segment $\overline{ab}$ and any neighborhood $\overline{ab} \subset U$, there is a measurable function $\zeta : U \to \{1,1+c\}$, $\epsilon > 0$ and a finite collection of non-intersecting horizontal segments $\overline{l_jr_j} \subset U$ in the direction of $x$ with the following property: if $a,b$ are $A$-stretched by some $L$-bi-Lipschitz map $f : U \to G$ whose Jacobian $\rho$ satisfies
  \begin{align*}
    |\zeta^{-1}(1) \backslash \rho^{-1}([0,1])| &< \epsilon, \\
    |\zeta^{-1}(1+c) \backslash \rho^{-1}([1+c,\infty))| &< \epsilon,
  \end{align*}
  then for some $j$ the pair $l_j, r_j$ is $(1+k')A$-stretched by $f$.
\end{lemma}

\begin{proof}
  We are allowed to translate and scale, and so we may assume $a = 0$ and $b = e^x$.  Let $\lambda \in (0,\lambda_0)$ so that $R_\lambda \subset U$.  Now choose $N > K \lambda^{-r}$ and set $\epsilon = \mu N^{-Q}$.  Choose a finite collection of disjoint horizontal segments in $R_\lambda$ using Lemma \ref{l:dense-segments} with $\eta' = 1/2N$ and $\delta = \gamma/4L^2N$ for some $\gamma > 0$ to be determined.  We then define
  \begin{align*}
    \zeta(p) = \begin{cases}
      \rho_{\lambda,N}(p), & p \in R_\lambda, \\
      1,  & p \in U \backslash R_\lambda.
    \end{cases}
  \end{align*}
  Defining $\epsilon = \mu N^{-Q}$ where $\mu > 0$ is from Lemma \ref{l:stretch-incr-1}, we get that there exist $\eta > 1/2N$ so that $g,ge^{\eta x} \in R_\lambda$ are $(1+k)A$-stretched.  There then exist $\overline{l_jr_j}$ so that $d(l_j,g) < \gamma/4L^2N$ and $d(r_j,g e^{\eta x}) < \gamma/4L^2N$.  We get that
  \begin{align*}
    d(l_j,r_j) = |\pi_x(l_j) - \pi_x(r_j)| \leq |\pi_x(g) - \pi_x(ge^{\eta x})| + 2\delta = d(g,ge^{\eta x}) + \frac{\gamma}{2L^2N}.
  \end{align*}
  On the other hand, two applications of Lemma \ref{l:quasi-triangle} gives
  \begin{align*}
    d(f(l_j),f(r_j)) \geq \frac{1}{(1+\zeta(\gamma))^2}d(f(g),f(ge^{\eta x})) \geq \frac{(1+k)A}{(1+\zeta(\gamma))^2} d(g,ge^{\eta x}).
  \end{align*}
  As $d(g,ge^{\eta x}) = \eta \geq \frac{1}{2N}$ and $\zeta(\gamma) \to 0$ as $\gamma \to 0$, we get by taking $\gamma$ sufficiently small relative to $N$ that
  \begin{align*}
    \frac{d(f(l_j),f(r_j))}{d(l_j,r_j)} \geq \frac{(1+k) A d(g,ge^{\eta x})}{(1+\zeta(\gamma))^2(d(g,ge^{\eta x}) + \gamma/2L^2N)} \geq \left( 1 + \frac{k}{2} \right) A.
  \end{align*}
  Thus, we finish the lemma by setting $k' = k/2$.
\end{proof}

We can now nest these Jacobians together to get the following lemma.  The proof is a straightforward modification Lemma 3.7 of \cite{BK}, being sure to use Lemma \ref{l:one-step}.

\begin{lemma}
  Let $\overline{ab}$ be a horizontal line segment in the $x$ direction and $U \supset \overline{ab}$ be an open neighborhood.  For each $i \geq 1$ there is a measurable function $\zeta_i : U \to \{1,1+c\}$, a finite collection $\mathcal{S}_i$ of non-intersection horizontal segments $\overline{l_jr_j} \subset U$ in the direction of $x$ and $\epsilon_i > 0$ with the following property: For every $L$-bi-Lipschitz map $f : U \to G$ whose Jacobian $\rho$ satisfies
  \begin{align*}
    |\zeta_i^{-1}(1) \backslash \rho^{-1}([0,1])| &< \epsilon_i, \\
    |\zeta_i^{-1}(1+c) \backslash \rho^{-1}([1+c,\infty))| &< \epsilon_i,
  \end{align*}
  at least one segment from $\mathcal{S}_i$ will have endpoints $\frac{(1+k)^i}{L}$-stretched by $f$.
\end{lemma}

The following theorem now easily follows:

\begin{thm} \label{t:non-jac}
  Let $G$ be a Carnot group that is not $\R$.  For any $c > 0$, there exists a function $\zeta : B(0,1) \to \{1,1+c\}$ such that there is no bi-Lipschitz map $f : B(0,1) \to G$ with Jacobian $\rho$ so that there is some $\lambda > 0$ for which
  \begin{align*}
    |\zeta^{-1}(1) \backslash \rho^{-1}([0,\lambda])| &= 0, \\
    |\zeta^{-1}(1+c) \backslash \rho^{-1}([(1+c)\lambda,\infty))| &= 0.
  \end{align*}
\end{thm}

\subsection{Constructing a net}

We now prove the following theorem from which Theorem \ref{thm:BLequiv} will easily follow.

\begin{thm} \label{t:bk}
  Let $H$ be a locally compact group of polynomial growth endowed with a left-invariant path metric and $Y \subset H$ be a separated net.  If the asymptotic cone of $H$ is not $\R$, then there exists another separated net $X$ of $H$ that is not bi-Lipschitz equivalent to $Y$.
\end{thm}

We can now construct the net to prove Theorem \ref{t:bk}.  Following \cite{BK}, the construction will be based on discretizing the Jacobian constructed in Theorem \ref{t:non-jac}.  However, we need to be a little careful as $Y$ is not necessarily a lattice and so asymptotically may not resemble the Lebesgue measure but just $\mu$ so that $\frac{1}{C} \cL^n \leq \mu \leq C \cL^n$.  This is why we needed to prove Theorem \ref{t:non-jac} with $\zeta$ as a threshold.  The fact that the non-Jacobian is for the asymptotic cone as opposed to the base space also introduces some complexity.

To construct $X$, we will need the following theorem of Christ that says for any doubling metric measure space that there exists a sequence of nested partitions that behaves much like dyadic cubes.
\begin{thm}[Theorem 11 of \cite{christ}] \label{t:cubes}
  Let $(X,d,\mu)$ be a doubling metric measure space.  There exists a collection of open subsets $\Delta := \{Q_\omega^k \subset G : k \in \Z, \omega \in I_k\}$ and constants $\eta > 0$, $\tau > 1$, $C_1 > 0$ such that
  \begin{enumerate}[(a)]
    \item $\mu\left( X \backslash \bigcup_\omega Q_\omega^k \right) = 0, \qquad \forall k \in \Z$.
    \item If $k \geq j$ then either $Q_\alpha^j \cap Q_\omega^k = \emptyset$ or $Q_\alpha^j \subset Q_\omega^k$.
    \item For each $(j,\alpha)$ and each $k > j$ there exists a unique $\omega$ so that $Q_\alpha^j \subset Q_\omega^k$.
    \item For each $Q_\omega^k$ there exists some $z_{Q_\omega^k}$ such that
    \begin{align}
      B(z_{Q_\omega^k}, \tau^k) \subseteq Q_\omega^k \subseteq B(z_{Q_\omega^k}, \tau^{k+2}). \label{e:quasiball}
    \end{align}
    \item $\mu\{x \in X \backslash Q^k_\alpha : d(x, Q_\alpha^k) \leq t \tau^k\} \leq C_1 t^\eta \mu(Q_\alpha^k) \qquad \forall k,\alpha, \forall t \in (0,1)$. \label{e:small-boundaries}
  \end{enumerate}
\end{thm}
We let $\Delta_k := \{Q_\omega^k : \omega \in I_k\}$.

\begin{remark}
  Property \eqref{e:small-boundaries} is called the ``small boundaries'' property and says that the boundary of a cube is small in the sense that small thickened (outer) neighborhoods of it have quantitatively vanishing measure.  In \cite{christ}, the small boundary property actually controlled the inner neighborhood
  \begin{align*}
    \mu\{x \in Q^k_\alpha, d(x,X \backslash Q_\alpha^k) \leq t \tau^k\} \leq C_1 t^\eta \mu(Q_\alpha^k).
  \end{align*}
  However, except for a possible null subset, the outer boundary of a cube lies in the union of the inner boundaries of the neighbors of the cubes.  As each cube has a uniformly bounded number of neighbor cubes of the same scale, the inner small boundaries estimate of \cite{christ} implies the above outer small boundaries.
\end{remark}

\begin{proof}[Proof of Theorem \ref{t:bk}]
  By Proposition 1.3 of \cite{breuillard}, we know that $H$ is $(1,C)$-quasi-isometric to $S$ for some $C \geq 1$ where $S$ is a connected and simply connected solvable Lie group of polynomial growth endowed with a periodic metric $d$.  We may suppose $Y$ satisfies
  \begin{align}
    d(x,y) \geq 100C, \qquad x \neq y \in Y. \label{e:100C-sep}
  \end{align}
  Indeed, take any separated net of $H$ satisfying \eqref{e:100C-sep}.  If this new net is bi-Lipschitz homeomorphic to $Y$, then proving Theorem \ref{t:bk} with $Y$ replaced by the new net also proves it for $Y$.  If the new net is not bi-Lipschitz homeomorphic to $Y$, then we are done.

  Because of \eqref{e:100C-sep}, we easily see that $Y$ is then bi-Lipschitz homeomorphic to its image $Y'$ in $S$ under the $(1,C)$-quasi-isometry, and so it suffices to find another separated net of $S$ that is not bi-Lipschitz homeomorphic to $Y'$ that also satisfies \eqref{e:100C-sep}.  Thus, we now replace $H$ by $S$ and $Y$ by $Y'$, which has separation at least $98C \geq 1$.

  By Theorem 1.4 of \cite{breuillard}, we know that there is another Lie group structure which makes $H$ into a Carnot group with left invariant (with respect to the new Lie structure) subFinsler metric $d_\infty$.  Call this new Lie group $G$.  Moreover the two metrics satisfy the following asymptotic convergence (see the claim in the proof of Corollary 1.9 of \cite{breuillard})
  \begin{align}
    \left| \frac{1}{n} d(x,y) - d_{\infty}(\delta_{1/n}(x),\delta_{1/n}(y))\right| \to 0 \label{e:GH-metr-conv}
  \end{align}
  as $n \to \infty$ for $x,y \in B_H(0,n)$ where $B_H(0,n)$ is the $n$-ball in the original Lie structure.

  \begin{claim} \label{cl:infinity-contain}
    For every $\epsilon > 0$, there exists $r > 0$ so that
    \begin{align*}
      B_G(0,(1-\epsilon)r) \subseteq B_H(0,r) \subseteq B_G(0,(1+\epsilon)r).
    \end{align*}
  \end{claim}

  \begin{proof}
    Suppose $B_G(0,(1-\epsilon)r) \not\subseteq B_H(0,r)$ for a sequence $r_n \to \infty$.  Then there exist $x_n \in B_G(0,(1-\epsilon)r_n)$ so that $x_n \notin B_H(0,r_n)$.  This means $d_H(0,x_n) > r_n$ but $d_G(0,x_n) < (1-\epsilon)r_n$.  This clearly violates \eqref{e:GH-metr-conv} with $x = x_n$ and $y = 0$.  The $B_H(0,r) \subseteq B_G(0,(1+\epsilon)r)$ follows similarly.
  \end{proof}
  
  From now on, we let $(H,d)$ and $(G,d_\infty)$ denote the two metrized Lie groups that share the same underlying manifold (which is homeomorphic to some $\R^n$).  We then let $\cL$ denote the Lebesgue measure on $G$ that is Haar for both $G$ and $H$.  Let $V_H(r) = |B_H(0,r)|$ and $V_G(r) = |B_G(0,r)|$.  We have that
  \begin{align}
    \lim_{r \to \infty} \frac{V_H(r)}{V_G(r)} \to 1. \label{e:vol-compare}
  \end{align}

  For any separated net $X \subset H$, we let $\mu_X$ denote the counting measure of $X$ on $H$.  Let $id : H \to G$ be the identity map of the underlying manifold between the two metrics and let $\phi_n = \delta_{1/n} \circ id : (H, n^{-1}d) \to (G,d_\infty)$
  
  \begin{claim} \label{cl:net-weak}
    For any $a \geq 1$, there exist $\beta > \alpha > 0$ so that if $X$ is a $(s,as)$-separated net of $H$ for $s \geq 1$ and $x_n \in H$, then there exists a subsequence of the measures
    \begin{align*}
      \mu_n := n^{-Q} ((\phi_n)_*\mu_{x_nX}) |_{B_G(0,1)}
    \end{align*}
    that weakly converges.  In addition, any such weak limit is of the form $g \cL$ where $g$ is a measurable function that satisfies
    \begin{align}
      s^{-Q} \alpha \leq g \leq s^{-Q} \beta, \label{e:g-bounds}
    \end{align}
    and may depend on the subsequence.
  \end{claim}

  \begin{proof}
    By \eqref{e:vol-compare}, there exists some $C \geq 1$ depending only on $G$ and $H$ so that
    \begin{align}
      \frac{1}{C} s^Q V_G(\lambda) \leq V_H(\lambda s) \leq C s^Q V_G(\lambda), \qquad \forall \lambda \geq 1. \label{e:v-growth}
    \end{align}

    A simple packing argument shows that for any $x,y \in H$ and $r$ sufficiently large (say $\geq 10s$), we have
    \begin{align}
      \frac{V_H(r)}{V_H(as)} = \frac{|B_H(x,r)|}{|B_H(0,as)|} \leq \# (yX \cap B_H(x,r)) \leq \frac{|B_H(x,2r)|}{|B_H(0,s)|} = \frac{V_H(2r)}{V_H(s)}. \label{e:net-count-bounds}
    \end{align}

    Let $B_G(x,r) \subseteq B_G(0,1)$ and $f : G \to [0,1]$ be a continuous function so that
    \begin{align}
      \chi_{B_G(x,r)} \leq f \leq \chi_{B_G(x,2r)}. \label{e:f-cutoff}
    \end{align}
    It follows from \eqref{e:GH-metr-conv} that for sufficiently large $n$ that 
    \begin{align}
      B_G(x,2r) \subset \phi_n(B_H(\phi_n^{-1}(x), 3rn)) \label{e:ball-in}
    \end{align}
    and so for all $n$ sufficiently large we have
    \begin{multline*}
       \int f ~d\mu_n \overset{\eqref{e:f-cutoff} \wedge \eqref{e:ball-in}}{\leq} n^{-Q} \# (x_nX \cap B_H(\phi_n^{-1}(x), 3rn)) \overset{\eqref{e:net-count-bounds}}{\leq} n^{-Q} \frac{V_H(6rn)}{V_H(s)} \\ \overset{\eqref{e:vol-compare}}{\leq} n^{-Q} \frac{2 V_G(6rn)}{V_H(s)} \overset{\eqref{e:v-growth}}{\leq} s^{-Q} \frac{C6^{Q+1}}{V_G(1)} V_G(r).
    \end{multline*}
    In particular, we have for every ball $B_G(x,r) \subset B_G(0,1)$ that $\mu_n(B_G(x,r)) \leq \int f ~d\mu_n$ is uniformly bounded for sufficiently large $n$ and so convergent subsequences of $\mu_n$ exist by, say, Banach-Steinhaus and Riesz representation theorem (see Theorem 11.4.7 of \cite{hkst} for a similar proof).  This establishes existence of a converging subsequence.

    For any weak limit $\mu$ of a subsequence of the $\mu_n$, we also have for all $B_G(x,r) \subseteq B_G(0,1)$ that
    \begin{align*}
      \mu(B_G(x,r)) \overset{\eqref{e:f-cutoff}}{\leq} \int f ~d\mu \leq s^{-Q} \frac{C6^{Q+1}}{V_G(1)} |B_G(x,r)|.
    \end{align*}
    Using a similar argument, one gets for all $B_G(x,r) \subseteq B_G(0,1)$ that
    \begin{align*}
      \mu(B_G(x,r)) \geq s^{-Q} \frac{1}{C6^{Q+1} V_G(a)} |B_G(x,r)|.
    \end{align*}
    Thus, we get \eqref{e:g-bounds} for $\alpha = \frac{1}{C6^{Q+1}V_G(a)}$ and $\beta = \frac{C6^{Q+1}}{V_G(1)}$.
  \end{proof}

  Using Theorem \ref{t:cubes}, we construct the dyadic cubes $\Delta^G$ and $\Delta^H$ for $G$ and $H$.  Note by \eqref{e:quasiball} that $X_j := \{z_Q : Q \in \Delta_j^H\}$ are $(\tau^j,\tau^{j+2})$-nets.  We can then apply Claim \ref{cl:net-weak} with $a = \tau^2$ to get $\beta_0 > \alpha_0 > 0$ so that any weak limit satisfies
  \begin{align}
    \alpha_0 \tau^{-jQ} \cL \leq \lim_n n^{-Q} ((\phi_n)_* \mu_{x_nX_j})|_{B_G(0,1)} \leq \beta_0 \tau^{-jQ} \cL \label{e:center-dense}
  \end{align}
  for $j \geq 0$.  We can also apply Claim \ref{cl:net-weak} to $Y$ to get that any weak limit satisfies
  \begin{align}
    \alpha_1 \cL \leq \lim_n n^{-Q} ((\phi_n)_* \mu_{x_nY})|_{B_G(0,1)} \leq \beta_1 \cL. \label{e:y-conv}
  \end{align}
  For simplicity, let us assume $\alpha_0 = \alpha_1 = 1$.  We can apply Theorem \ref{t:non-jac} with $c = \beta_0 - 1$ to get a function $\zeta : B_G(0,1) \to \{1,\beta_0\}$ that serves as a threshold for Jacobians of bi-Lipschitz maps.
  
  For each $k > 0$, let
  \begin{align*}
    \cT_k := \{Q \in \Delta_{-k}^G : Q \subset B_G(0,1)\}.
  \end{align*}
  We also define $\cT_k' = \left\{Q \in \cT_k : \fint_Q \zeta < \frac{1+\beta_0}{2} \right\}$ and $R_k = \bigcup_{Q \in \cT_k'} Q$.

  \begin{claim} \label{cl:meas-conv}
    For $A_k \subset B_G(0,1)$ so that $R_k \subseteq A_k \subseteq B_G(R_k, k^{-1} \tau^{-k})$, we have that $A_k$ converges in measure to $\zeta^{-1}(1)$.
  \end{claim}

  This is an easy measure theoretic argument, and we will prove the claim at the end of this subsection.  For now, let us assume the claim and construct the net $X$.

  Now let $B_{j,k} := B_H(x_{j,k},r_k)$ be a collection of disjoint balls in $H$ where $r_k \to \infty$.  For each $j$, we have that $B_{j,k}$ is isometric to $B_H(0,r_k)$ via the left translations by $x_{j,k}^{-1}$.
  
  We now construct $X$.  Fix a $j_0 \in \mathbb{N}$ so that
  \begin{align}
     \tau^{-j_0Q} \leq \frac{1}{\beta_0^2 \beta_1} \label{e:double-density}
  \end{align}
  and a $j_1 \in \mathbb{N}$ so that $\tau^{j_1} \geq 100C$.  For each $R \in \Delta^H_{j_0+j_1}$, if there exists $j,k$ so that $R \subset B_{j,k}$ and $d_\infty(\phi_{r_k}(x_{j,k}^{-1}R),R_j) < j^{-1} \tau^{-j}$, add the point $z_R$ to $X$.  Otherwise, add the the points $\{z_Q : Q \in \Delta_{j_1}^H, Q \subset R\}$.  Altogether, we get a $100C$-separated net $X \subset H$.

  Now suppose there is a $L$-bi-Lipschitz homeomorphism $f : X \to Y$.  Let $X_{j,k} := X \cap B_{j,k}$, $z_{j,k}$ be some arbitrary basepoint in $f(X_{j,k})$ and $Y_{j,k} := z_{j,k}^{-1} Y \cap B_H(0,2Lk)$.  We can define then the functions
  \begin{align*}
    f_{j,k} : X_{j,k} &\to Y_{j,k} \\
    x &\mapsto z_{j,k}^{-1} f(x).
  \end{align*}

  We now establish some convergences along subsequences.  By the construction of $X$, \eqref{e:center-dense}, the assumption $\alpha_0 = 1$, the choice of $j_0$, and \eqref{e:double-density}, for every $j$, we get the following convergence in $k$ along a subsequence (which we have reindexed along)
  \begin{align*}
    (r_k)^{-Q} ((\phi_{j,k})_* \mu_{x_{j,k}^{-1}X})|_{B_G(0,1)} \rightharpoonup \frac{1}{h_j} \cL.
  \end{align*}
  Here, for each $j$ there is a measurable set $A_j$ so that $R_j \subseteq A_j \subseteq B_G(R_j,j^{-1}\tau^j)$ and $h_j$ is a measurable function satisfying
  \begin{align*}
    |A_j \backslash h_j^{-1}([0,1])| &= 0, \\
    |(B_G(0,1) \backslash A_j) \backslash h_j^{-1}([\beta_0\beta_1,\infty))| &= 0.
  \end{align*}
  Taking Claim \ref{cl:meas-conv} into account, we can get by a diagonalization argument that there is the following convergence along a subsequence of $j$
  \begin{align}
    \mu_j := r_{k_j}^{-Q} ((\psi_{j,k_j})_* \mu_{x_{j,k_j}^{-1}X})|_{B_G(0,1)} \rightharpoonup \frac{1}{h} \cL \label{e:x-conv}
  \end{align}
  where $h$ is a measurable function satisfying
  \begin{align}
    |\zeta^{-1}(1) \backslash h^{-1}([0,1])| &= 0, \label{e:h-1} \\
    |\zeta^{-1}(\beta_0) \backslash h^{-1}([\beta_0 \beta_1,\infty))| &= 0. \label{e:h-beta}
  \end{align}
  
  Let $\Phi_j : B_G(0,1) \to (X_{j,k_j},r_{k_j}^{-1}d)$ be the composition of the inverse of $id$, the left translation by $x_{j,k_j}$, and the nearest point mapping to $X_{j,k_j}$.  Also let $\Psi_j : (Y_{j,k_j},r_{k_j}^{-1}d) \to B_G(0,2L)$ be the composition of left translation by $z_{j,k_j}^{-1}$ with $\phi_{r_{k_j}}$.  Note that $\Phi_j$ and $\Psi_j$ are $\epsilon_j$-isometries where $\epsilon_j \to 0$ where $\epsilon_j$ depends on the rate of convergence of \eqref{e:GH-metr-conv}.  By construction, we have that $(\Phi_j)_*\mu_j = r_{k_j}^{-Q} \mu_{X_{j,k_j}}$.
  
  As all the $f_{j,k_j}$ are $L$-bi-Lipschitz, the function $F_j = \Psi_{j,k_j} \circ f_{j,k_j} \circ \Phi_{j,k_j} : B_G(0,1) \to G$ satisfies the quasi-isometry conditions
  \begin{align*}
    \frac{1}{L}d(x,y) - (L+1) \epsilon_j \leq d(F_j(x),F_j(y)) \leq Ld(x,y) + (L+1) \epsilon_j.
  \end{align*}
  As $\epsilon_j \to 0$ and $F_j(B_G(0,1)) \subseteq B_G(0,2L)$, we can use an Arzela-Ascoli argument to get a subsequence that ``uniformly converges'' to a $L$-bi-Lipschitz map $f : B_G(0,1) \to G$.
  
  We also have by construction that
  $$(F_j)_* \mu_j = r_{k_j}^{-Q} (\phi_{r_{k_j}})_* \mu_{z_{j,k_j}^{-1}f_{j,k_j}(X_{j,{k_j}})}.$$
  By \eqref{e:y-conv}, we have that $(F_j)_* \mu_j$ converges weakly to $g\cL|_{f(B_G(0,1))}$ where (remembering we've assumed $\alpha_1 = 1$)
  \begin{align*}
    1 \leq g \leq \beta_1.
  \end{align*}
  This together with \eqref{e:x-conv} gives that $f_*\left( \frac{1}{h} \cL \right) = g\cL|_{f(B_G(0,1))}$ and so $f$ has Jacobian $\rho = h/g$.  As $g \in [1,\beta_1]$, we get
  \begin{align*}
    |\zeta^{-1}(1) \backslash \rho^{-1}([0,1])| &\overset{\eqref{e:h-1}}{=} 0, \\
    |\zeta^{-1}(\beta_0) \backslash \rho^{-1}([\beta_0,\infty))| &\overset{\eqref{e:h-beta}}{=} 0,
  \end{align*}
  a contradiction of the defining properties of $\zeta$.

  Let us now prove Claim \ref{cl:meas-conv}.

  \begin{proof}[Proof of Claim \ref{cl:meas-conv}]
    Let $A = \zeta^{-1}(1) \subset B_G(0,1)$, a measurable set that we may suppose without loss of generality has positive measure.  By the small boundary properties of the cubes, it suffices to show that $R_k$ converges to $A$ in measure.  Let $\epsilon > 0$.  We will drop the $G$ subscript because we will never use $H$ in this proof.
    
    Fix some $\epsilon > 0$.  By the Lebesgue density theorem, there exists $R > 0$ and $A' \subseteq A$ so that $|A'| > (1-\epsilon) |A|$ and
    \begin{align*}
      |B(x,r) \cap A| \geq (1-\epsilon)|B(x,r)|, \qquad \forall x \in A', r \in (0,R].
    \end{align*}
    Let $N$ be so that $\tau^{-k+1} < R$ for all $k \geq N$.  An easy density argument using \eqref{e:quasiball} and the definition of $A'$ then gives that for all $k \geq N$ and $Q \in \Delta_{-k}$ so that $A' \cap Q \neq \emptyset$, we have
    \begin{align}
      |Q| \leq \frac{|Q \cap A|}{1-c(\epsilon)} \label{e:dense-bound}
    \end{align}
    for some $c(\epsilon)$ depending only on $G$ and $\epsilon$ so that $\lim_{\epsilon \to 0} c(\epsilon) = 0$.  
    
    We will show
    \begin{align}
      |A \triangle R_k| \leq 2(\epsilon + c(\epsilon)) |A|, \qquad \forall k \geq N. \label{e:c-eps-bound}
    \end{align}
    This would prove the claim.  Thus, fix any $k \geq N$.

    The inequality \eqref{e:dense-bound} gives that there is some $\epsilon_0 > 0$ so that if $\epsilon < \epsilon_0$ and $Q \in \Delta_{-k}$ is so that $Q \cap A' \neq \emptyset$, then $Q \in \cT_k'$.  Thus, we get that $R_k \supseteq A'$ and so \eqref{e:c-eps-bound} follows if we can show $|R_k| < (1+2 \epsilon + 2c(\epsilon))|A|$.  We can assume also that $\epsilon < \epsilon_0$.

    Define
    \begin{align*}
      \cT_1 := \{Q \in \Delta_{-k} : A' \cap Q \neq \emptyset\} \text{ and } \cT_2 := \cT_k' \backslash \cT_1.
    \end{align*}
    Note that $A \backslash \bigcup_{Q \in \cT_1} Q \supseteq A \backslash A'$ and so $\left|A \cap \bigcup_{Q \in \cT_2} Q\right| < \epsilon |A|$.  However, by definition of $\cT_k'$, we get that
    \begin{align}
      \sum_{Q \in \cT_2} |Q| \leq 2 \sum_{Q \in \cT_2} |A \cap Q| \leq 2 \epsilon |A|. \label{e:T2-bound}
    \end{align}
    
    Now we have
    \begin{align*}
      |R_k| = \sum_{Q \in \cT_1} |Q| + \sum_{Q \in \cT_2} |Q| \overset{\eqref{e:dense-bound} \wedge \eqref{e:T2-bound}}{\leq} \frac{|A|}{1-c(\epsilon)} + 2\epsilon |A| \leq (1 + 2 c(\epsilon)  + 2\epsilon) |A|.
    \end{align*}
    This verifies \eqref{e:c-eps-bound} and finishes the proof of the claim.
  \end{proof}
\end{proof}

\bibliographystyle{plain}
\bibliography{BLBD}

\def\cprime{$'$} \def\cprime{$'$} \def\cprime{$'$}
\begin{thebibliography}{10}

\bibitem{BG}
Michael Baake and Uwe Grimm.
\newblock {\em Aperiodic Order}, volume~1.
\newblock Cambridge University Press, 2013.

\bibitem{hartnicketal}
M.~{Bj{\"o}rklund}, T.~{Hartnick}, and F.~{Pogorzelski}.
\newblock {Aperiodic order and spherical diffraction}.
\newblock {\em ArXiv e-prints}, February 2016.

\bibitem{breuillard}
Emmanuel Breuillard.
\newblock Geometry of locally compact groups of polynomial growth and shape of
  large balls.
\newblock {\em Groups Geom. Dyn.}, 8(3):669--732, 2014.

\bibitem{BK}
D.~Burago and B.~Kleiner.
\newblock Separated nets in {E}uclidean space and {J}acobians of bi-{L}ipschitz
  maps.
\newblock {\em Geom. Funct. Anal.}, 8(2):273--282, 1998.

\bibitem{BK2002}
Dmitri Burago and Bruce Kleiner.
\newblock Rectifying separated nets.
\newblock {\em Geom. Funct. Anal.}, 12(1):80--92, 2002.

\bibitem{christ}
M.~Christ.
\newblock A $t(b)$ theorem with remarks on analytic capacity and the {C}auchy
  integral.
\newblock {\em Colloq. Math.}, 60-61(2):601--628, 1990.

\bibitem{deBruijn}
NG~De~Bruijn.
\newblock Algebraic theory of {P}enrose's non-periodic tilings of the plane. i.
\newblock {\em Indagationes Mathematicae (Proceedings)}, 84(1):39--52, 1981.

\bibitem{scoop}
Michel Duneau and Christophe Oguey.
\newblock Displacive transformations and quasicrystalline symmetries.
\newblock {\em J. Physique}, 51(1):5--19, 1990.

\bibitem{DO2}
Michel Duneau and Christophe Oguey.
\newblock Bounded interpolations between lattices.
\newblock {\em J. Phys. A}, 24(2):461--475, 1991.

\bibitem{DN16}
Tullia Dymarz and Andres Navas.
\newblock Non-rectifiable {D}elone sets in sol and other solvable groups.
\newblock {\em To appear in Indiana Univ. Math. J.}, 154, 2016.

\bibitem{GreenTao}
Ben Green and Terence Tao.
\newblock The quantitative behaviour of polynomial orbits on nilmanifolds.
\newblock {\em Ann. of Math. (2)}, 175(2):465--540, 2012.

\bibitem{gromov1996geometric}
Mikhail Gromov.
\newblock Geometric group theory, vol. 2: Asymptotic invariants of infinite
  groups.
\newblock 1996.

\bibitem{HKW}
Alan Haynes, Michael Kelly, and Barak Weiss.
\newblock Equivalence relations on separated nets arising from linear toral
  flows.
\newblock {\em Proc. Lond. Math. Soc. (3)}, 109(5):1203--1228, 2014.

\bibitem{hkst}
Juha Heinonen, Pekka Koskela, Nageswari Shanmugalingam, and Jeremy~T. Tyson.
\newblock {\em Sobolev spaces on metric measure spaces}, volume~27 of {\em New
  Mathematical Monographs}.
\newblock Cambridge University Press, Cambridge, 2015.
\newblock An approach based on upper gradients.

\bibitem{Kuranishi}
Masatake Kuranishi.
\newblock On everywhere dense imbedding of free groups in {L}ie groups.
\newblock {\em Nagoya Math. J.}, 2:63--71, 1951.

\bibitem{Laczkovich92UnifSpread}
Mikl{\'o}s Laczkovich.
\newblock Uniformly spread discrete sets in {${\bf R}^d$}.
\newblock {\em J. London Math. Soc. (2)}, 46(1):39--57, 1992.

\bibitem{Li}
Sean Li.
\newblock Coarse differentiation and quantitative nonembeddability for {C}arnot
  groups.
\newblock {\em J. Funct. Anal.}, 266(7):4616--4704, 2014.

\bibitem{magazinov}
A.~N. Magazinov.
\newblock The family of bi-{L}ipschitz classes of {D}elone sets in {E}uclidean
  space has the cardinality of the continuum.
\newblock {\em Tr. Mat. Inst. Steklova}, 275(Klassicheskaya i Sovremennaya
  Matematika v Pole Deyatelnosti Borisa Nikolaevicha Delone):87--98, 2011.

\bibitem{Malcev}
A.~I. Mal{\cprime}cev.
\newblock On a class of homogeneous spaces.
\newblock {\em Izvestiya Akad. Nauk. SSSR. Ser. Mat.}, 13:9--32, 1949.

\bibitem{mcmullen1998lipschitz}
Curtis~T McMullen.
\newblock Lipschitz maps and nets in euclidean space.
\newblock {\em Geometric and Functional Analysis}, 8(2):304--314, 1998.

\bibitem{S1}
Lorenzo Sadun.
\newblock {\em Topology of tiling spaces}, volume~46 of {\em University Lecture
  Series}.
\newblock American Mathematical Society, Providence, RI, 2008.

\bibitem{Strichartz}
Robert~S. Strichartz.
\newblock Self-similarity on nilpotent {L}ie groups.
\newblock In {\em Geometric analysis ({P}hiladelphia, {PA}, 1991)}, volume 140
  of {\em Contemp. Math.}, pages 123--157. Amer. Math. Soc., Providence, RI,
  1992.

\bibitem{Whyte}
Kevin Whyte.
\newblock Amenability, bi-{L}ipschitz equivalence, and the von {N}eumann
  conjecture.
\newblock {\em Duke Math. J.}, 99(1):93--112, 1999.

\end{thebibliography}

\end{document}